\documentclass[11pt, letterpaper]{amsart}
\usepackage{amsmath}
\usepackage{amssymb}
\usepackage[left=1in,right=1in,bottom=0.8in,top=0.8in]{geometry}
\usepackage{amsfonts}
\usepackage{graphicx}
\usepackage[font=small,labelfont=bf]{caption}
\usepackage{epstopdf}
\usepackage[pdfpagelabels,hyperindex]{hyperref}
\usepackage{xcolor}
\usepackage{amsthm}
\usepackage{float}
\usepackage{pgfplots}
\usepackage{listings}
\usepackage{longtable}
\usepackage{mathrsfs}
\usepackage{bbm}
\usepackage{euscript}
\usepackage{mathtools}
\usepackage{tikz}
\usepackage{comment}
\usepackage{tikz-cd}
\usepackage{amsxtra}
\usepackage{pifont}
\usepackage{amsbsy}
\usepackage{wasysym}
\usepackage{calligra}
\usepackage{url}
\usepackage{enumerate}
\usepackage{xcolor}
\usepackage[
backend=biber,
style=alphabetic,
]{biblatex}
\newtheorem{theorem}{Theorem}[section]
\newtheorem{corollary}[theorem]{Corollary}
\newtheorem{proposition}[theorem]{Proposition}
\newtheorem{lemma}[theorem]{Lemma}
\newtheorem{conjecture}[theorem]{Conjecture}

\theoremstyle{definition}
\newtheorem{definition}[theorem]{Definition}
\newtheorem{construction}[theorem]{Construction}
\newtheorem{ex}[theorem]{Example}
\newtheorem{notation}[theorem]{Notation}

\theoremstyle{remark}
\newtheorem{remark}[theorem]{Remark}

\definecolor{energy}{RGB}{114,0,172}
\definecolor{freq}{RGB}{45,177,93}
\definecolor{spin}{RGB}{251,0,29}
\definecolor{signal}{RGB}{203,23,206}
\definecolor{circle}{RGB}{217,86,16}
\definecolor{average}{RGB}{203,23,206}

\colorlet{shadecolor}{gray!20}
\pgfplotsset{compat=1.9}

\usepgflibrary{fpu}

\newcommand{\RR}{\mathbb{R}}
\newcommand{\ZZ}{\mathbb{Z}}

\newcommand{\GG}{\mathbb{G}}

\newcommand{\PP}{\mathbb{P}}

\newcommand{\bR}{\mathbf{R}}

\newcommand{\cA}{\mathcal{A}}

\newcommand{\cC}{\mathcal{C}}

\newcommand{\cE}{\mathcal{E}}
\newcommand{\cF}{\mathcal{F}}
\newcommand{\cG}{\mathcal{G}}

\newcommand{\cL}{\mathcal{L}}
\newcommand{\cM}{\mathcal{M}}
\newcommand{\cN}{\mathcal{N}}
\newcommand{\cO}{\mathcal{O}}
\newcommand{\cP}{\mathcal{P}}

\newcommand{\sC}{\mathscr{C}}
\newcommand{\sI}{\mathscr{I}}
\newcommand{\sX}{\mathscr{X}}

\newcommand{\Aut}{\operatorname{Aut}}
\newcommand{\Br}{\operatorname{Br}}

\newcommand{\cha}{\operatorname{char}}

\newcommand{\Coh}{\operatorname{Coh}}

\newcommand{\End}{\operatorname{End}}
\newcommand{\cEnd}{\mathscr{E}\text{\kern -3pt {\calligra\large nd}}\,}
\newcommand{\et}{\operatorname{\acute{e}t}}

\newcommand{\Ext}{\operatorname{Ext}}

\newcommand{\Gal}{\operatorname{Gal}}
\newcommand{\GL}{\operatorname{GL}}

\newcommand{\Hilb}{\operatorname{Hilb}}
\newcommand{\Hom}{\operatorname{Hom}}
\newcommand{\cHom}{\mathscr{H}\text{\kern -3pt {\calligra\large om}}\,}
\newcommand{\id}{\operatorname{id}}
\newcommand{\ind}{\operatorname{ind}}

\newcommand{\Mix}{\operatorname{Mix}}

\newcommand{\Norm}{\operatorname{Norm}}

\newcommand{\per}{\operatorname{per}}
\newcommand{\Pic}{\operatorname{Pic}}
\newcommand{\Prym}{\operatorname{Prym}}

\newcommand{\PGL}{\operatorname{PGL}}

\newcommand{\Spec}{\operatorname{Spec}}

\newcommand{\tw}{\operatorname{tw}}
\newcommand{\univ}{\operatorname{univ}}

\author[Ting Gong]{Ting Gong}
\email{tgong2@uw.edu}
\title[moduli of twisted vector bundles]{Moduli of vector bundles on $\mu_n$-gerbes over genus 2 curves and the period-index problem}
\begin{document}
\maketitle

\begin{abstract}
We develop a framework for describing vector bundles on $\mu_n$-gerbes over curves and illustrate the construction through two detailed examples.  Using the interpretation of Brauer classes as obstructions to descending determinantal line bundles from the algebraic closure, together with a geometric analysis of the moduli space of twisted sheaves, we prove that for genus $2$ curves there exist Brauer classes over the base field whose period equals their index.  Over $C_1$-fields, we further show that every $2$-torsion class in the Brauer group of a genus-$2$ curve satisfies the period–index problem.  As an application, we construct higher-dimensional varieties obtained as fibre products of genus $2$ curves over $C_1$-fields whose $2$-torsion algebraic Brauer classes also satisfy the period–index problem, providing new evidence toward the period–index conjecture.
\end{abstract}

\section{Introduction}
\subsection{Moduli problems}
The moduli of vector bundles on curves has been a central topic in algebraic geometry since its early development by authors such as Seshadri, Ramanan, Narasimhan, Mumford, and Newstead. Explicit computations for low-rank moduli spaces were carried out in the works of Narasimhan–Ramanan \cite{Narasimhan_Ramanan_69}, Newstead \cite{Newstead_68}, and Desale–Ramanan \cite{Desale_Ramanan_76}, among others. These computations revealed a close relationship between the moduli space of semistable vector bundles of fixed rank and determinant and certain universal line bundles on the moduli space, notably the Theta divisor. This perspective was further developed in the works of Beauville, Narasimhan, and Ramanan \cite{Beauville_88, Beauville_Narasimhan_Ramanan_89, Beauville_94}.

When the base field is not algebraically closed, in studying the moduli problem of Azumaya algebras, Lieblich \cite{Lieblich_07} constructed the moduli space of twisted sheaves and applied it to several cases of the period–index problem, thereby extending earlier work such as that of Căldăraru \cite{Caldararu_00} and Yoshioka \cite{Yoshioka_06}. 

In this work, we revisit the developments discussed above and apply their techniques to explicitly describe the moduli space of semistable twisted vector bundles of rank~$2$ with trivial determinant as twisted linear systems associated to the generalized Theta divisor in the sense of Kollár~\cite{Kollar_16}. Our approach proceeds by studying the determinantal line bundle on twisted moduli spaces. In particular, we show that 

\begin{theorem}\label{Intro1}
    Let $k$ be a field of characteristic not equal to $2$, and let $C$ be a smooth projective curve of genus~$2$ over~$k$. Let $\alpha \in H^2(C, \mu_2)$ be a $\mu_2$-gerbe such that $\alpha_{\bar{k}} = 0$. Then the moduli space of semistable rank-$2$ vector bundles with trivial determinant twisted by $\alpha$, denoted $M^{ss,\alpha}_C(2,\cO_C)$, is isomorphic to the twisted linear system $|\Theta|$. In particular, $M^{ss,\alpha}_C(2,\cO_C)$ is a Brauer–Severi variety over $k$.
\end{theorem}

\begin{remark}
    We note that \cite{Iyer_Parimala_22} provides a description of the moduli space of twisted sheaves and establishes the period--index result only for locally essentially trivial Brauer classes.  Our Theorems~\ref{Intro1} and~\ref{intro5} are more general, as they apply to all Brauer classes.
\end{remark}

Having an explicit description of the moduli space allows us to classify vector bundles of fixed rank and degree on $\mu_n$-gerbes over curves. These moduli spaces naturally decompose into unions of components corresponding to moduli spaces of untwisted, twisted, and mixed vector bundles of the same rank and degree. As concrete illustrations, we compute the cases of rank-$2$ vector bundles on $\mu_2$-gerbes and rank-$3$ vector bundles on $\mu_3$-gerbes. Note that these computations work over any smooth projective varieties, and we restrict them to curves. 

\begin{theorem}
Let $k$ be a field of characteristic not equal to $2$, and let $C$ be a smooth projective curve of genus~$2$ over~$k$. Let $\alpha \in H^2(C, \mu_2)$ represent a $\mu_2$-gerbe $\sC \to C$. Then the moduli space $M^{ss}_{\sC}(2, \cO_C)$ of semistable rank $2$ vector bundles with trivial determinant has two connected components. If $\alpha_{\bar{k}} = 0$, then 
\[
M^{ss}_{\sC}(2, \cO_C) \cong |\Theta| \cong \PP^3 \cup X,
\]
where $X$ is either $\PP^3$ or its Brauer–Severi twist. If $\alpha_{\bar{k}} \neq 0$, then 
\[
M^{ss}_{\sC}(2, \cO_C) \cong |\Theta| \cong \PP^3 \cup X,
\]
where $X$ is either the intersection of two quartic hypersurfaces in $\PP^5$ or its twist.
\end{theorem}

In general, when classifying rank-$r$ vector bundles with fixed determinant on a $\mu_n$-gerbe, not every bundle arises purely as an untwisted bundle or as a twisted bundle.  Mixed phenomena can occur, in which the twisting behavior is distributed among several characters of $\mu_n$.  We include an explicit example illustrating such a mixed situation.

\begin{theorem}
Let $k$ be an arbitrary field, and let $X$ be a smooth projective variety over~$k$. Let $\sX \to X$ be a $\mu_3$-gerbe representing a class $\alpha \in H^2(X, \mu_3)$ with $\alpha_{\bar{k}} = 0$. Then
\[
M^{ss}_{\sX}(3, \cO_X) \;\cong\;M^{ss}_{X}(3, \cO_X) \; \coprod \; M^{ss,\alpha}_{X} (3, \cO_X)\;\coprod\;M^{ss,\alpha^2}_{X}(3, \cO_X)\;\coprod\;\Mix(X,3,\cO_X),
\]
where $\Mix(X,3,\cO_X)$ is a subspace of \(\Pic^0_X \oplus \Pic^0_{X,\alpha} \oplus \Pic^0_{X,\alpha^2}\) whose closed points correspond to triples \((\cL, \cL_\alpha, \cL_{\alpha^2})\) satisfying $\cL \otimes \cL_\alpha \otimes \cL_{\alpha^2} \cong \cO_X$ (up to a finite extension of the base field~$k$).
\end{theorem}

\subsection{The period–index problem}

In studying the Brauer group, there are two numerical invariants attached to a Brauer class $\alpha \in \Br(X)$ for a variety $X$ over a field $k$: the \emph{period}, denoted $\per(\alpha)$, which is the order of $\alpha$ in $\Br(X)$, and the \emph{index}, denoted $\ind(\alpha)$, which is the minimal degree of a field extension $K/k$ such that $\alpha_K = 0$. It is well known that $\per(\alpha)$ divides $\ind(\alpha)$ and that the two share the same prime factors \cite[Theorem~2.8.7]{Gille_Szamuely_2006}. Hence, $\ind(\alpha)$ divides a power of $\per(\alpha)$. The question of whether this exponent depends on the variety $X$ is the classical \emph{period–index problem}, first posed by Colliot-Thélène~\cite{Colliot-Thélène_02}.

\begin{conjecture}[The period-index problem]
    If $K$ is a $C_d$ field, then for all $\alpha\in \Br(K)$, we have $\ind(\alpha)$ divides $\per(\alpha)^{d-1}$. 
\end{conjecture}

The conjecture has been established in numerous cases.  For the reader’s convenience, we collect below the results known to the author.
\begin{itemize}
    \item When $d = 0$, the field $K$ is algebraically closed, and the statement is trivial since $\Br(K) = 0$.  
    \item When $d = 1$, Tsen's theorem \cite[Theorem~6.2.8]{Gille_Szamuely_2006} implies that $\Br(K) = 0$, again verifying the conjecture.  
    \item When $d = 2$, de~Jong, Lieblich, and Starr, in a sequence of papers \cite{De_Jong_04, Lieblich_08, De_Jong_Starr_10}, showed that if $L$ is a field of transcendence degree~$2$ over an algebraically closed field $K$, then $\ind(\alpha) = \per(\alpha)$ for all $\alpha \in \Br(L)$.
    \item When $d = 3$, very recently Hotchkiss and Perry~\cite{Hotchkiss_Perry_24} proved that the period–index conjecture holds for Abelian threefolds over algebraically closed fields.
\end{itemize}

The period–index problem has also been studied in a variety of arithmetic and geometric settings.  
\begin{itemize}
    \item Lieblich showed in \cite{Lieblich_08, Lieblich_11} that if $L$ is a finitely generated field of transcendence degree~$2$ over a local field $K$, with $L$ an $\mathcal O_K$-algebra, and if $\alpha \in \Br(L)$ has period prime to the residue characteristic of $K$, then:
    \begin{itemize}
        \item if $\alpha$ is unramified, then $\ind(\alpha)=\per(\alpha)$;
        \item if $\alpha$ is ramified, then $\ind(\alpha)\mid \per(\alpha)^2$.
    \end{itemize}
    \item Antieau, Auel, Ingalls, Krashen, and Lieblich proved in \cite{Antieau_Auel_Ingalls_Krashen_Lieblich_19} that if $S$ is a geometrically integral surface over a $p$-adic field $K$ and $\alpha \in \Br(K(S))$ has period prime to $6p$, then $\ind(\alpha) \mid \per(\alpha)^4$.
    \item Benoist~\cite{Benoist_19} showed that if $S$ is a connected smooth projective surface over $\RR$ with $S(\RR)=\varnothing$, and $\alpha\in \Br(\RR(S))$, then $\ind(\alpha) = \per(\alpha)$.
    \item Huybrechts showed in \cite{Huybrecht_25} that if $X$ is a projective hyperkähler manifold admitting a Lagrangian fibration, then there exists an integer $N_X$ such that for every $\alpha \in \Br(X)$ with $\per(\alpha)$ coprime to $N_X$ one has $\ind(\alpha) \mid  \per(\alpha)^{\dim(X)/2}$.

    In particular, if $X = S^{[n]}$ is the Hilbert scheme of a $K3$ surface and $\alpha\in \Br(X)$, then $\ind(\alpha) \mid \per(\alpha)^n$.
\end{itemize}

In this paper, as an application of the explicit description of the moduli space of twisted sheaves of rank~$2$ with trivial determinant as a Brauer--Severi variety, we compute the period and index of the corresponding Brauer class in terms of the index of the Brauer--Severi variety and the existence of a rational point. We first establish a general result for the period--index problem over an arbitrary field for a genus $2$ curve.

\begin{theorem}
    Let $C$ be a smooth projective curve of genus $2$ over an arbitrary field $k$ with $\operatorname{char} k \neq 2$. Suppose $\alpha \in \Br(C)[2]$ is a $2$-torsion Brauer class. Assume further that for every finite extension $K/k$ and every $\beta \in \Br(K)$, one has $\ind(\beta) \mid 2^{\,i}$. Then $\ind(\alpha) \mid 2^{\,i+2}$.
\end{theorem}

In particular, we prove that for a genus-$2$ curve over a $C_1$-field, every $2$-torsion Brauer class satisfies the period-index equality. 

\begin{theorem}\label{intro5}
Let $k$ be a $C_1$-field of characteristic not equal to $2$, and let $C/k$ be a smooth projective curve of genus~$2$. Then for every $\alpha \in \Br(C)[2]$, one has $\per(\alpha) = \ind(\alpha) = 2$.
\end{theorem}

Using the above theorem, we construct an $n$-fold over a $C_1$-field whose algebraic $2$-torsion Brauer classes satisfy the period–index problem, thereby providing further evidence for the conjecture.

\begin{theorem}
Let $k$ be a $C_1$-field of characteristic not equal to $2$, and let $X = C_1 \times \cdots \times C_n$ be the product of $n$ smooth projective curves of genus~$2$ over $k$. Then for every $\alpha \in \Br_1(X)[2]$, the period–index problem holds; that is, $\ind(\alpha) \mid \per(\alpha)^{\,n} = 2^{n}$.
\end{theorem}

\subsection{Structure of the paper}
Section~2 recalls the classical constructions in the theory of gerbes and twisted sheaves.  
Section~3 develops the determinantal line bundle on the moduli space of twisted sheaves and shows that this moduli space is the good moduli space associated to its moduli stack, with semistability defined via the determinantal line bundle.  We then provide explicit computations of the moduli of vector bundles on gerbes over curves.  
Section~4 gives explicit descriptions of the moduli spaces of twisted rank $2$ vector bundles with trivial determinant, both in the essentially trivial case and in the optimal cases.  
Finally, Section~5 establishes the period–index result.

\subsection{Notations and conventions}
Throughout, unless otherwise specified, we work over a field $k$ of characteristic $p \ge 0$ with $\operatorname{char}(k)\ne 2$. When considering $\mu_{n}$-gerbes, we additionally assume that $n$ is invertible in $k$. We let $X$ denote an arbitrary smooth projective variety over $k$, and $C$ a geometrically smooth projective curve over $k$. All cohomology groups are taken with respect to the étale topology.

A $k$-rational point on $C$ is required only when we invoke the splitting of the Leray spectral sequence; in all other arguments, no rational point on $C$ is assumed.

\section*{Acknowledgment}
The author would like to thank his advisor, Max Lieblich, for suggesting this topic and for his numerous discussions, time, and patience. The author also thanks James Hotchkiss for his careful reading of earlier drafts, extensive feedback, and many helpful discussions. Finally, the author thanks Enhao Feng for reading earlier drafts and providing feedback, and Jarod Alper, Yu Shen, and Bianca Viray for their helpful discussions.

\section{Review of gerbes and twisted sheaves}
In this section, we review the theory of gerbes and twisted sheaves over a smooth projective variety, together with basic properties of the moduli stack of twisted sheaves.  We then specialize to the case of the moduli stack of twisted sheaves over a curve of genus $\ge 2$, where the notion of degree can be made completely explicit.  Throughout this section, unless otherwise specified, we continue to use the notation introduced in the notation and conventions subsection.

\subsection{Gerbes and twisted sheaves}
We begin by recalling some basic definitions and results concerning $G$-gerbes and twisted sheaves.

\begin{definition}
    Let $G$ be an abelian sheaf on the small étale site $X_{\et}$. A \emph{$G$-gerbe} on a scheme $X$ is a gerbe $\sX \to X$ (on $X_{\et}$) equipped with an isomorphism $G_{\sX} \xrightarrow{\;\sim\;} \sI(\sX),$ where $\sI(\sX)$ denotes the inertia stack of $\sX$.
\end{definition}

\begin{theorem}[{\cite[Theorem~12.2.8]{Olsson_16}}]
    Let $\sC$ be a site, and let $G$ be an abelian sheaf on $X$. Then the set of isomorphism classes of banded $G$-gerbes over $\sC$ is in natural bijection with the cohomology group $H^2(\sC, G)$.
\end{theorem}

In our setting, the group sheaf $G$ will be either $\GG_m$ or $\mu_n$. Accordingly, the Brauer group $\Br(X)$ and the cohomology group $H^2(X,\mu_n)$ classify $\GG_m$-gerbes and $\mu_n$-gerbes over $X$, respectively. With a mild abuse of notation, we shall henceforth use the terms $\GG_m$-gerbe, $\mu_n$-gerbe, and the corresponding cohomology class in $\Br(X)$ or $H^2(X,\mu_n)$ interchangeably.

Next, we recall the definitions of twisted sheaves following \cite{Lieblich_07} and \cite{Caldararu_00}. As the following proposition shows, these definitions are equivalent.

\begin{definition}
    Let $\sX\rightarrow X$ be a $G$-gerbe, a \emph{$\sX$-twisted sheaf} is a sheaf $\cF$ of $\cO_{\sX}$-modules such that the inertial action $\cF\times_{\sX} \sI(\sX)\rightarrow \cF$ equals the right action associated to the left module action $\sI(\sX)\times_{\sX} \cF\rightarrow \cF$. 
\end{definition}

\begin{definition}
    Let $X$ be a scheme and let $\alpha \in H^2(X, G)$. Choose an étale covering $\{U_i \to X\}$ and a Čech $2$-cocycle \(\alpha_{ijk} \in \Gamma(U_i \times_X U_j \times_X U_k, G)\) representing $\alpha$. An \emph{$\alpha$-twisted sheaf} on $X$ consists of a collection \((\cF_i, \varphi_{ij})\) where each $\cF_i$ is a quasi-coherent $\cO_{U_i}$-module, and for each $i,j$, $\varphi_{ij} : \cF_i|_{U_{ij}} \xrightarrow{\sim} \cF_j|_{U_{ij}}$ are isomorphisms satisfying the cocycle condition $\varphi_{jk} \circ \varphi_{ij} = \alpha_{ijk} \cdot \varphi_{ik}\text{ on } U_{ijk},$ where $\alpha_{ijk}$ acts on $\cF_i$ via the $G$-action.
\end{definition}

\begin{proposition}[{\cite[Lemma~2.10]{De_Jong_05}}]
    There is an equivalence between the category of $\alpha$-twisted sheaves, defined via a Čech $2$-cocycle representing $\alpha \in H^2(X,G)$, and the category of $\sX$-twisted sheaves on the corresponding $G$-gerbe $\sX \to X$ representing $\alpha$.
\end{proposition}

\begin{proposition}[{\cite[Proposition~1.2.10]{Caldararu_00}}]\label{operation}
    Let $\cF$ be an $\alpha$-twisted sheaf and $\cG$ a $\beta$-twisted sheaf on $X$. Then: $\cF \otimes \cG$ is an $\alpha \beta$-twisted sheaf, and $\cHom(\cF, \cG)$ is an $\alpha^{-1}\beta$-twisted sheaf.
\end{proposition}

\begin{remark}
    The equality $\Br(X) = \Br'(X)$ between the Azumaya and cohomological Brauer groups is not known in general. However, it is known to hold for quasi-projective schemes \cite{De_Jong_05}. Under this assumption, by \cite{Edidin_Hassett_Kresch_Vistoli_01}, every cohomological Brauer class corresponds to an Azumaya algebra, i.e. each class in $\Br(X)$ lies in the image of the Brauer map $H^1(X, \PGL_n) \rightarrow H^2(X, \GG_m).$
\end{remark}

\begin{proposition}[{\cite[Proposition~3.1.2.1]{Lieblich_08}}]\label{existtwist}
    Let $\alpha \in \Br(X)$ be a Brauer class. Then $\alpha$ lies in the image of the Brauer map $H^1(X, \PGL_n) \rightarrow H^2(X, \GG_m)$ if and only if there exists a locally free $\alpha$-twisted sheaf on $X$ of rank $n$ with trivial determinant.
\end{proposition}

\subsection{Essentially trivial gerbes}
In this subsection, we study the special case of essentially trivial gerbes. Twisted sheaves arising from such gerbes behave in many ways like ordinary sheaves, and their moduli stacks retain much of the structure familiar from the classical untwisted theory.

\begin{definition}
    Let $\alpha \in H^2(X, \mu_n)$ be a $\mu_n$-gerbe. We say that $\alpha$ is \emph{essentially trivial} if its image under the natural map $H^2(X, \mu_n) \rightarrow H^2(X, \GG_m)$ is zero. Conversely, $\alpha$ (or the corresponding $\mu_n$-gerbe) is said to be \emph{optimal} if the associated $\GG_m$-gerbe has order exactly $n$ in $H^2(X, \GG_m)$.
\end{definition}

\begin{ex}
    Let $\cL$ be a line bundle on $X$. The \emph{$n$th root gerbe} of $\cL$ is the stack whose objects over a scheme $T$ are pairs $(\cM, \varphi)$, where $\cM$ is a line bundle on $X \times T$ and $\varphi : \cM^{\otimes n} \xrightarrow{\;\sim\;} \cL_T$ is an isomorphism of line bundles. This is an essentially trivial $\mu_n$-gerbe on $X$.
\end{ex}

\begin{lemma}[{\cite[Lemma~2.3.4.2]{Lieblich_07}}]\label{esstriv}
    A $\mu_n$-gerbe $f \colon \sX \to X$ is essentially trivial if and only if there exists an invertible $\sX$-twisted sheaf $\cL$ on $\sX$.
\end{lemma}

\begin{proposition}[{\cite[Proposition~2.3.4.4]{Lieblich_07}}]\label{Kummer}
    The cohomology of the Kummer exact sequence
    \[
    1 \longrightarrow \mu_n \longrightarrow \GG_m \xrightarrow{(\cdot)^n} \GG_m \longrightarrow 1
    \]
    yields an exact sequence
    \[
    0 \longrightarrow \Pic(X)/n\Pic(X) \xrightarrow{\;\partial\;} H^2(X, \mu_n) \longrightarrow H^2(X, \GG_m)[n] \longrightarrow 0.
    \]
    Then the cohomology class of any essentially trivial $\mu_n$-gerbe $\alpha \in H^2(X,\mu_n)$ equals the image under $\partial$ of the class of a line bundle of the form $\cL^{\otimes n}$, where $\cL$ is an invertible $\sX$-twisted sheaf. Consequently, every essentially trivial $\mu_n$-gerbe is the $n$th-root gerbe of some line bundle on~$X$.
\end{proposition}

\begin{construction}
    Let $\sX \to X$ be a $\mu_n$-gerbe representing a class $\alpha \in H^2(X,\mu_n)$. Following \cite{Knudsen_Mumford_76}, there is a determinant morphism
    \[
        \det \colon \cM_{\sX} \longrightarrow \Pic_{\sX/k},
    \]
    where $\cM_{\sX}$ denotes the moduli stack of vector bundles on $\sX$.

    If $\cF$ is a locally free $\alpha$-twisted sheaf of rank $m$ on $\sX$, then its determinant $\det(\cF)$ is an $\alpha^m$-twisted line bundle. In particular, when $m$ is equal to $n$, the class $\alpha^n$ is trivial, so $\det(\cF)$ is an untwisted line bundle on $X$.  Thus the determinant morphism induces a map
    \[
        \det \colon \cM^{\alpha}_{X}(n) \longrightarrow \Pic_{X/k},
    \]
    where $\cM^{\alpha}_{X}(n)$ is the moduli stack of rank $n$ $\alpha$-twisted vector bundles on $X$.

    For a line bundle $\cL$ on $X$, we write $\cM^{\alpha}_{X}(n,\cL)$ for the fibre of the determinant morphism over $\cL$, i.e.\ the moduli stack of rank $n$ $\alpha$-twisted vector bundles on $X$ with determinant $\cL$.
\end{construction}

The following lemma illustrates that twisted vector bundles on essentially trivial gerbes behave much like ordinary vector bundles.

\begin{lemma}\label{ln=l}
    If $\alpha \in H^2(X, \mu_n)$ is essentially trivial, then there is an equivalence of abelian categories between $\Coh(X)$ and $\Coh(X, \alpha)$. Moreover, for suitable line bundles $\cN$ and $\cL$ on $X$, there is an isomorphism of moduli stacks $\cM_X(n, \cN)\; \xrightarrow{\;\simeq\;}\; \cM^{\alpha}_X(n, \cN \otimes \cL).$
\end{lemma}

\begin{proof}
    Let $\alpha \in H^2(X, \mu_n)$ be essentially trivial, and let $\cL_\alpha$ be an invertible $\alpha$-twisted sheaf as in Lemma~\ref{esstriv}. By Proposition~\ref{Kummer}, we have $\cL_\alpha^{\otimes n} \cong \cL$ for some line bundle $\cL$ on $X$. Then the claimed equivalence of categories is given by tensoring with $\cL_\alpha$:
    \[
    (-) \otimes \cL_\alpha \colon \Coh(X) \;\longrightarrow\; \Coh(X, \alpha),
    \]
    whose inverse is given by tensoring with $\cL_\alpha^{-1}$.

    Let $\cF$ be a vector bundle on $X$ of rank $n$ and determinant $\cN$. Applying the functor above, we obtain an $\alpha$-twisted vector bundle $\cF \otimes \cL_\alpha$ of rank $n$ and determinant \(\det(\cF \otimes \cL_\alpha) \cong \cN \otimes \cL.\) This induces the desired isomorphism of moduli stacks stated in the lemma.
\end{proof}

\begin{remark}
    In general, the situation is more subtle. Suppose there exists an $\alpha$-twisted locally free sheaf $\cF$ for some $\alpha \in H^2(X, \mu_n)$ not necessarily essentially trivial. Then, as shown in \cite[Proposition~1.3.6]{Caldararu_00}, there is an equivalence of categories between $\Coh(X, \alpha)$ and $\Coh(X, \cA)$, where $\Coh(X, \cA)$ denotes the category of coherent right $\cA$-modules on~$X$, and $\cA = \End(\cF)$ is the Azumaya algebra associated to $\alpha$. The equivalence is given by tensoring with $\cF^{\vee}$, as in the essentially trivial case; the distinction lies in the fact that here $\End(\cF) \cong \cA$ is nontrivial.
\end{remark}

\subsection{Character decomposition}
In classifying sheaves on a $\mu_n$-gerbe, one naturally encounters a decomposition according to the characters of $\mu_n$. In this subsection we record several instances of such character decompositions for different kinds of objects.

\begin{theorem}[{\cite[Corollary~4.11]{Bergh_Schnurer_19}}]\label{decomp}
    Let $X$ be an algebraic stack, and let $\sX \to X$ be a $\mu_n$-gerbe classified by a cohomology class $\alpha \in H^2(X, \mu_n)$. Then there is an equivalence of abelian categories
    \[
    \Coh(\sX) \;\xrightarrow{\;\sim\;}\;\prod_{i \in \ZZ/n\ZZ} \Coh(X, \alpha^i).
    \]
\end{theorem}

\begin{definition}
    Let $\alpha \in H^2(X, \mu_n)$. The \emph{$\alpha$-twisted Grothendieck group} $K_0^\alpha(X)$ is defined as the quotient of the free abelian group generated by all $\alpha$-twisted vector bundles on~$X$ by the subgroup generated by all elements of the form $[V] - [V'] - [V'']$ whenever there is an exact sequence 
    \[
    0 \longrightarrow V' \longrightarrow V \longrightarrow V'' \longrightarrow 0
    \]
    of $\alpha$-twisted vector bundles on~$X$.
\end{definition}

\begin{proposition}\label{decompofK}
    Let $\alpha \in H^2(X,\mu_n)$ and let $\sX \to X$ be the associated $\mu_n$-gerbe. Then the Grothendieck ring of $\sX$ is isomorphic to the $\ZZ/n\ZZ$–graded Grothendieck ring of $\alpha^i$-twisted sheaves on $X$:
    \[
    K_0(\sX)\;\cong\;\bigoplus_{i=0}^{n-1} K_0^{\alpha^i}(X).
    \]
\end{proposition}

\begin{proof}
    The bijection follows from Theorem~\ref{decomp}. Since $K_0$ is defined via the Grothendieck group of the abelian category of coherent sheaves, this induces the additive isomorphism. The ring structure is given by tensor product, and by Proposition~\ref{operation} the twist satisfies \(\alpha^i \cdot \alpha^j = \alpha^{i+j}\), which gives the claimed grading.
\end{proof}

\begin{remark}
    The decomposition in Proposition~\ref{decompofK} is classical and appears in related forms in the literature.  In particular, an analogous decomposition for the Grothendieck group of a Brauer–Severi variety is proved in \cite[Theorem~4.1]{Quillen_73}.
\end{remark}

\subsection{Degrees of twisted sheaves on curves}
In this subsection, we define the degree of a twisted vector bundle on a curve and explain how this notion interacts with essentially trivial $\mu_n$-gerbes. We remark that, in contrast to the case of curves, the notion of degree is far more subtle on a smooth projective variety of dimension $\ge 2$, where no canonical degree map exists without the choice of a polarization.

\begin{definition}
    Let $\alpha \in H^2(C,\mu_n)$, and let $\cF$ be a locally free $\alpha$-twisted sheaf of rank $n$ on the curve $C$. We define the \emph{degree} of $\cF$ by
    \[
        \deg(\cF) := \frac{1}{n}\,\deg\!\bigl(c_1(\det(\cF))\bigr) \in \mathbb{Q},
    \]
    where $\det(\cF)$ is an $\alpha^n$-twisted line bundle, hence an actual line bundle on $C$.
    
    In particular, if $\cL$ is an $\alpha$-twisted line bundle such that $\cL^{\otimes n}$ is a genuine line bundle on $C$, then
    \[
        \deg(\cL) = \frac{1}{n}\deg\!\bigl(\cL^{\otimes n}\bigr).
    \]
\end{definition}

\begin{notation}
    We denote by $\Pic_\alpha(C)$ (resp.~$\Pic^i_\alpha(C)$) the set of $\alpha$-twisted line bundles (resp.~of degree~$i$) on~$C$. We refer to $\Pic_\alpha(C)$ (resp.~$\Pic^i_\alpha(C)$) as the \emph{$\alpha$-twisted Picard set} (resp.~the \emph{degree~$i$ $\alpha$-twisted Picard set}) of~$C$.
\end{notation}

\begin{proposition}
    Assume that $\sC \to C$ is an essentially trivial $\mu_n$-gerbe. Then the degree map $\deg \colon \Pic(\sC) \longrightarrow \tfrac{1}{n}\ZZ$ is a group homomorphism.
\end{proposition} 

\begin{proof}
    Let $\alpha \in H^2(C, \mu_n)$ be an essentially trivial class representing the gerbe $\sC$. Then each power $\alpha^i$ is also essentially trivial for all $i \in \ZZ / \per(\alpha)\ZZ$. By Lemma~\ref{esstriv}, there exist invertible $\alpha^i$-twisted sheaves on $C$ for all such $i$.  

    Applying Theorem~\ref{decomp} and restricting to the subcategory of invertible sheaves, we obtain a decomposition
    \[
    \Pic(\sC) \;=\; \coprod_{i \in \ZZ / n\ZZ} \Pic_{\alpha^i}(C),
    \]
    since line bundles have rank~$1$.

    Assume that $\cL$ and $\cM$ are $\alpha^i$- and $\alpha^j$-twisted line bundles, respectively. Since $\alpha$ is essentially trivial, we have $\per(\alpha) = n$ in $H^2(C,\mu_n)$. Hence $\cL^{\otimes n}$ and $\cM^{\otimes n}$ are honest line bundles on $C$, say of degrees
    \[
    \deg(\cL^{\otimes n}) = i + c_1 n, \qquad\deg(\cM^{\otimes n}) = j + c_2 n,
    \]
    for some integers $c_1, c_2$.  

    Then $(\cL \otimes \cM)^{\otimes n}\;\cong\;\cL^{\otimes n} \otimes \cM^{\otimes n},$ which has degree
    \[
    \deg((\cL \otimes \cM)^{\otimes n}) = (i + c_1 n) + (j + c_2 n)= (i + j) + (c_1 + c_2) n.
    \]
    Therefore,
    \[
    \deg(\cL \otimes \cM)= \frac{i + j}{n} + (c_1 + c_2)= \deg(\cL) + \deg(\cM),
    \]
    showing that the degree map $\deg \colon \Pic(\sC) \to \frac{1}{n}\ZZ$ is a group homomorphism.
\end{proof}

\begin{proposition}
    The set $\Pic_\alpha(C)$ is nonempty if and only if $\alpha \in H^2(C, \mu_n)$ is essentially trivial. In this case, choose an $\alpha$-twisted line bundle $\cL_\alpha$ on $C$. Then tensoring with $\cL_\alpha^{\vee}$ induces a natural bijection
    \[
        \Pic_{\alpha}^i(C) \;\xrightarrow{\;\sim\;}\; 
        \Pic^{\,i - d}(C), \qquad 
        \cM_\alpha \longmapsto \cM_\alpha \otimes \cL_\alpha^{\vee},
    \]
    where $d = \deg(\cL_\alpha) \in \tfrac{1}{n}\mathbb{Z}$. In particular, the degree of an $\alpha$-twisted line bundle need not be integral.
\end{proposition}

\begin{proof}
    This follows directly from Lemmas~\ref{esstriv} and ~\ref{ln=l}.
\end{proof}

\begin{proposition}[{\cite[Lemma~3.1.2.1]{Lieblich_07}}]\label{esstrivdeg}
    Let $\alpha \in H^2(C, \mu_n)$ be an essentially trivial class. Then there exists a natural surjective homomorphism $\varphi \colon H^2(C,\mu_n) \rightarrow \ZZ/n\ZZ$ which sends an essentially trivial $\mu_n$-gerbe to the degree modulo~$n$ of the $n$th power of any $\alpha$-twisted line bundle. Consequently, there exists an invertible $\alpha$-twisted sheaf on $C$ of degree $\tfrac{1}{n}\,\varphi(\alpha).$
\end{proposition}

\begin{proof}
    Since $\alpha$ is essentially trivial, Proposition~\ref{Kummer} provides an invertible $\alpha$-twisted sheaf $\cL$ such that $\cL^{\otimes n}$ is an ordinary line bundle on $C$. Assume the degree of $\cL^{\otimes n}$ to be $m\in \ZZ$. Because the map $\Pic(C)/n\Pic(C) \hookrightarrow H^2(C,\mu_n)$ is injective, we may define \(\varphi(\alpha) := m \bmod n \in \ZZ/n\ZZ.\)

    Now write \(\deg(\cL^{\otimes n}) = \varphi(\alpha) + cn\) for some integer \(c\). Hence we have $\deg(\cL) = \tfrac{1}{n}\varphi(\alpha) + c.$ Tensoring $\cL$ by a line bundle of degree $-c$ yields an $\alpha$-twisted line bundle of degree $\tfrac{1}{n}\varphi(\alpha)$, as claimed.
\end{proof}

\section{Moduli theory of vector bundles on $\mu_n$-gerbes over curves}
In this section, we begin by introducing the determinantal line bundle on moduli stacks of twisted vector bundles with fixed rank and determinant. We then recall the existence theorem for good moduli spaces of twisted sheaves on curves, as well as that of $\mu_n$-gerbes over curves. Finally, we present our approach to classifying moduli spaces of vector bundles on $\mu_n$-gerbes over curves, together with the examples promised in the introduction.

Throughout this section, we work over a field $k$ of characteristic $p \ge 0$, where we assume $p \ne 2$. Let $C$ be a smooth projective curve over $k$. Whenever $\mu_n$-gerbes are involved, we additionally assume $\gcd(n,p)=1$, so that $\mu_n$ is a finite étale group scheme over $k$.

\subsection{Determinantal line bundle and semistability}
\begin{construction}
    Let $\alpha \in H^2(C,\mu_n)$ be a $\mu_n$-gerbe. We have the diagram
    \[
    \begin{tikzcd}
        & C \times \cM^\alpha_{C/k}(n,d)
            \arrow[dl, "q"'] 
            \arrow[dr, "p"] & \\
        C & & \cM^\alpha_{C/k}(n,d),
    \end{tikzcd}
    \]
    where $q$ and $p$ denote the projections to the first and second factors, respectively.
    
    In the twisted setting, there is no universal vector bundle on $C \times \cM^\alpha_{C/k}(n,d)$, but there exists a universal $\alpha$-twisted sheaf $\cE_{\alpha,\univ}$ of rank $n$ and degree $d$ on $C \times \cM^\alpha_{C/k}(n,d)$.

    Let $V$ be an $\alpha$-twisted vector bundle on $C \times \cM^\alpha_{C/k}(n,d)$. We define the \emph{determinantal line bundle} on $\cM^\alpha_{C/k}(n,d)$ associated to $V$ by
    $$\cL_{\alpha,V}: = (\det\bR p_*(\cHom (V, \cE_{\alpha,\univ})))^\vee = (\det\bR p_*(V^\vee\otimes  \cE_{\alpha,\univ}))^\vee, $$
    where $p : C \times \cM^\alpha_{C/k}(n,d) \to \cM^\alpha_{C/k}(n,d)$ is the projection.

    Since $V^\vee \otimes \cE_{\alpha,\univ}$ is an untwisted coherent sheaf on $C \times \cM^\alpha_{C/k}(n,d)$ whose restriction to each fiber over $\cM^\alpha_{C/k}(n,d)$ is a vector bundle on $C$, the derived pushforward $\mathbf{R}p_*(V^\vee \otimes \cE_{\alpha,\univ})$ is a perfect complex of amplitude contained in $[0,1]$. Thus its determinant is a line bundle, and $\cL_{\alpha,V}$ is a well-defined line bundle on $\cM^\alpha_{C/k}(n,d)$.

    Assume that for every $k$-rational point $[E] \in \cM_C^{ss,\alpha}(n,d)(k)$ one has $\chi\bigl(\cHom(V_E,E)\bigr) = 0,$ where $V_E := V|_{C \times \{[E]\}}$. Equivalently, for each such $[E]$,
    \[
    \dim H^0\bigl(C,\cHom(V_E,E)\bigr)= \dim H^1\bigl(C,\cHom(V_E,E)\bigr).
    \]
    In this case the perfect complex $\mathbf{R}p_*\cHom(V,\cE_{\alpha,\univ})$ has rank zero on $\cM_C^{ss,\alpha}(n,d)$, and the determinantal line bundle $\cL_{\alpha,V}$ admits a canonical global section $s_V \in \Gamma\bigl(\cM_C^{ss,\alpha}(n,d), \cL_{\alpha,V}\bigr)$, which is locally given by the determinant of the differential in a two-term complex of locally free sheaves quasi-isomorphic to $\mathbf{R}p_*\cHom(V,\cE_{\alpha,\univ})$ \cite[\href{https://stacks.math.columbia.edu/tag/0FJI}{Tag 0FJI}]{stacks-project}.

    Finally, if $[E]$ is a closed point of $\cM^{\alpha}_{C/k}(n,d)$ and $V = q^*W$ for some vector bundle $W$ on $C$, where $q : C \times \cM^{\alpha}_{C/k}(n,d) \to C$ is the projection, then the fiber of $\cL_{\alpha,V}$ at $[E]$ is given by
    \[
        \cL_{\alpha,V}|_{[E]}\;\cong\;\det\bigl(H^0\bigl(C,\cHom(W,E)\bigr)\bigr)^{\vee}\otimes\det\bigl(H^1\bigl(C,\cHom(W,E)\bigr)\bigr).
    \]
    
    In what follows, we restrict to the case where $V = q^*W$ is the pullback of an $\alpha$-twisted vector bundle $W$ on $C$.
\end{construction}

\begin{lemma}\label{multdet}
    Given a short exact sequence of $\alpha$-twisted vector bundles on $C \times \cM^\alpha_{C/k}(n,d)$
    \[
    0 \longrightarrow V' \longrightarrow V \longrightarrow V'' \longrightarrow 0,
    \]
    the associated determinantal line bundles satisfy the multiplicative relation
    \[
    \cL_V \;\cong\; \cL_{V'} \otimes \cL_{V''}.
    \]
\end{lemma}

\begin{proof}
    Applying $\cHom(-,\cE_{\alpha,\univ})$ to the short exact sequence yields an exact sequence of vector bundles
    \[
    0 \longrightarrow \cHom(V'',\cE_{\alpha,\univ})\longrightarrow \cHom(V,\cE_{\alpha,\univ})\longrightarrow \cHom(V',\cE_{\alpha,\univ})\longrightarrow 0.
    \]
    By the multiplicativity of determinants of perfect complexes (\cite[Proposition~4.2]{Alper_Belmans_Bragg_Liang_Tajakka_22}), we obtain the desired isomorphism $\cL_V \;\cong\; \cL_{V'} \otimes \cL_{V''}.$
\end{proof}

\begin{proposition}\label{detline}
    We have the following properties of determinantal line bundles:
    \begin{enumerate}
        \item[(i)] The Picard group $\Pic(\cM^\alpha_C(r,\cL))$ is isomorphic to $\ZZ$. 
        \item[(ii)] The assignment $V \mapsto \cL_V$ is additive in short exact sequences of $\alpha$-twisted vector bundles, and therefore induces a group homomorphism 
        \[
            K_0^\alpha(C)\longrightarrow \Pic(\cM^\alpha_C(r,d)).
        \]
        Hence the isomorphism class of $\cL_V$ depends only on the rank and determinant of $V$.
        \item[(iii)] If $V$ and $W$ are $\alpha$-twisted vector bundles of the same rank and degree, then there exists a line bundle $\cN$ on $\Pic^d(C)$ such that $\cL_W \cong \cL_V \otimes \det^*\cN.$ In particular, if $\det(V)=\det(W)$, then $\cL_W \cong \cL_V$.
    \end{enumerate}
\end{proposition}

\begin{proof}
    (i) By the Leray spectral sequence, we have an injection
    \[
    \Pic(\cM^\alpha_C(r,\cL)) \hookrightarrow \Pic(\cM^\alpha_C(r,\cL)_{\bar{k}}).
    \]
    Over $\bar{k}$ the gerbe $\alpha_{\bar{k}}$ is trivial, and by Lemma~\ref{ln=l} the stack $\cM^\alpha_C(r,\cL)_{\bar{k}}$ is isomorphic to the usual moduli stack of vector bundles with fixed determinant. The description of its Picard group as $\ZZ$, is proved in \cite{Drezet_Narasimhan_89}; the argument applies verbatim over any algebraically closed field.

    (ii) By Lemma~\ref{multdet}, the assignment $V \mapsto \cL_V$ is multiplicative in short exact sequences of $\alpha$-twisted vector bundles. Hence it descends to a group homomorphism
    \[
    K_0^\alpha(C)\;\longrightarrow\; \Pic(\cM^\alpha_C(r,d)).
    \]
    
    To see that $\cL_V$ depends only on $\operatorname{rk}(V)$ and $\det(V)$, recall from part~(i) that $\Pic(\cM^\alpha_C(r,d))\hookrightarrow \Pic\bigl(\cM^\alpha_C(r,d)_{\bar{k}}\bigr)$ is injective. Over $\bar{k}$ the gerbe becomes trivial and $\cM^\alpha_C(r,d)_{\bar{k}}$ identifies with the usual moduli space of vector bundles with fixed determinant. On this untwisted moduli space, the determinantal line bundle depends only on the rank and determinant of $V$ by \cite[Proposition 4.5 (i)]{Alper_Belmans_Bragg_Liang_Tajakka_22}. Hence the same dependence holds over $k$.
    
    (iii) The argument is the same as in part~(ii). Over $\bar{k}$ the gerbe is trivial, so $\cM^\alpha_C(r,d)_{\bar{k}}$ is the usual moduli space of vector bundles with fixed determinant. By \cite[Proposition~4.5(ii)]{Alper_Belmans_Bragg_Liang_Tajakka_22}, if $V$ and $W$ have the same rank and degree, then $\cL_W \;\cong\; \cL_V \otimes \det^*\cN$ for some line bundle $\cN$ on $\Pic^d(C)$; in particular, if $\det(V)=\det(W)$, then $\cL_W\cong \cL_V$. By the injection of Leray spectral sequence (as in part~(i)), the same isomorphism holds over $k$.
\end{proof}

\begin{remark}
    In \cite{Heinloth_17}, the fact that $\Pic(\cM_C(r,\cL))=\ZZ$ implies that every line bundle on $\cM_C(r,\cL)$ differs from the determinantal line bundle by a tensor power, and therefore the resulting notion of $\cL$–stability is independent of the choice of $\cL$ and agrees with the usual slope stability.
    
    By Proposition~\ref{detline}, the same statement holds for the twisted moduli stack: the group $\Pic(\cM_C^\alpha(r,\cL))$ is isomorphic to $\ZZ$.  Hence for \emph{any} line bundle $\cL'$ on $\cM_C^\alpha(r,\cL)$, the corresponding $\cL'$–stability notion coincides with slope stability, exactly as in the untwisted case.
\end{remark}

\begin{definition}
    Let $\cF$ be a locally free $\alpha$-twisted sheaf on $C$. Its \emph{slope} is defined by
    \[
    \mu(\cF) := \frac{\deg(\cF)}{\operatorname{rk}(\cF)}.
    \]
    We say that $\cF$ is \emph{stable} (resp.\ \emph{semistable}) if for every proper $\alpha$-twisted subbundle $0 \neq \cG \subsetneq \cF$ one has $\mu(\cG) < \mu(\cF) \text{ (resp.\ $\mu(\cG) \le \mu(\cF)$).}$
\end{definition}

\subsection{Good moduli spaces for twisted vector bundles}
In this subsection, we record that the moduli stack of semistable $\alpha$-twisted vector bundles admits a good moduli space in the sense of \cite{Alper_13}.  The existence of this moduli space was first established by Lieblich in \cite{Lieblich_07} using GIT.

When working with good moduli spaces, the arguments are essentially the same as in \cite{Alper_Halpern-Leinstner_Heinloth_13} and \cite{Alper_Belmans_Bragg_Liang_Tajakka_22}.  In particular, once the determinantal line bundle is fixed, the proofs carry over verbatim to the twisted setting.  For this reason we state the results without reproducing the proofs, and refer the reader to the above references for details.

\begin{theorem}
    Let $\alpha \in H^2(C,\mu_r)$ be the class of a $\mu_r$–gerbe $\cC \to C$, and fix a line bundle $\cL$ of degree $d$. Then:
    \begin{itemize}
        \item[(i)] The stacks $\cM^{ss,\alpha}_{C/k}(r,d)$ and $\cM^{ss,\alpha}_{C/k}(r,\cL)$ are smooth integral algebraic stacks, locally of finite presentation over $k$ with affine diagonal, $\Theta$-reductive, $S$-complete, and universally closed. Consequently, they admit projective good moduli spaces  (of dimension $r^2(g - 1)+1$ and $(r^2 -1)(g - 1)$ respectively) in the sense of~\cite{Alper_13}.
        \item[(ii)] The good moduli space $M^{ss,\alpha}_C(r,\cL)$ is geometrically unirational, and the morphism on stable loci $\cM^{s,\alpha}_{C_{\bar{k}}}(r,\cL_{\bar{k}})\rightarrow M^{s,\alpha}_{C_{\bar{k}}}(r,\cL_{\bar{k}})$ is a $\mu_r$–gerbe. Likewise, the stable locus of $\cM_C^{ss,\alpha}(r,d)$ is (geometrically) a $\GG_m$-gerbe over its good moduli space.
        \item[(iii)] If $\gcd\!\left(r,\; r d - \varphi(\alpha)\right)=1,$ then every semistable $\alpha$-twisted vector bundle of rank $r$ and degree $d$ is stable. Consequently, the open immersion $\cM^{s,\alpha}_C(r,\cL)\hookrightarrow\cM^{ss,\alpha}_C(r,\cL)$ is an isomorphism, and the same conclusion holds for $\cM^{ss,\alpha}_C(r,d)$.
    \end{itemize}
\end{theorem}

\begin{proof}
    For part (i), algebraicity, local finite presentation, and affine diagonal follow from \cite[Proposition~2.3.1.1]{Lieblich_07}. Universal closedness follows from Langton’s semistable reduction theorem \cite{Langton_75}, as extended to the twisted setting in \cite[Lemma~2.3.3.2]{Lieblich_07}. $\Theta$-reductivity and $S$-completeness follow by the same arguments as in \cite[Proposition~3.8]{Alper_Belmans_Bragg_Liang_Tajakka_22} and \cite[Lemma~3.13]{Weissmann_Zhang_25}. The only additional input in the twisted setting is that, after possibly replacing the base field (or the DVR) by a finite extension, there exists an $\alpha_{\kappa}$-twisted locally free sheaf of rank \(r\) and degree \(d\). Once such a bundle exists, the diagrams appearing in the definitions of $\Theta$-reductivity and $S$-completeness can be filled exactly as in the untwisted case using the same techniques. The existence of projective good moduli spaces then follows from \cite[Theorem~5.4]{Alper_Halpern-Leinstner_Heinloth_13}.

    The remaining assertions follow from \cite[Corollary~3.2.2.4]{Lieblich_08}, together with the observation that neither the argument for (i) nor the gcd criterion in (iii) requires the base field to be algebraically closed, while the statements in (ii) are geometric and therefore become true after base change to~$\bar{k}$.
\end{proof}

\begin{remark}
    We remark that the structural results for moduli of twisted vector bundles parallel those for the usual moduli of vector bundles, and the proofs carry over with only minor modifications.  For further details, we refer the reader to related works such as \cite{Weissmann_Zhang_25} and \cite{Herrero_Weissmann_Zhang_25}.
\end{remark}

\begin{proposition}
    Assume there exists an $\alpha$-twisted vector bundle $V$ on $C$. Then the determinantal line bundle $\cL_V$ on the moduli stack $\cM^{ss,\alpha}_C(n,\cL)$ descends to its good moduli space $M^{ss,\alpha}_C(n,\cL)$ and the descended line bundle is semi-ample.
\end{proposition}
\begin{proof}
    By the same argument as in \cite[Lemma~3.11(ii)]{Alper_Belmans_Bragg_Liang_Tajakka_22}, applied in the abelian category of $\alpha$-twisted coherent sheaves on $C$ (see \cite{Lieblich_07} for the existence of Harder--Narasimhan and Jordan--Hölder filtrations in this setting), closed points of $\cM^{ss,\alpha}_C(r,d)$ correspond to polystable $\alpha$-twisted bundles.  Thus any closed point can be represented by $E \;=\; \bigoplus_{j=1}^n E_j^{\oplus m_j},$ where the $E_j$ are pairwise non-isomorphic stable $\alpha$-twisted bundles of rank $r_j$ and degree $d_j$ with slope $d_j/r_j = d/r$.

    We first show that the stabilizer of $[E]$ acts trivially on the fiber
    of $\cL_V$ at $[E]$.  The automorphism group of $E$ is isomorphic to
    \[
        \Aut(E) \;\cong\; \GL_{m_1}(\cA_{1})\times \cdots\times \GL_{m_n}(\cA_{n}),
    \]
    where each $\cA_j$ is a finite-dimensional Azumaya algebra.
    The fiber of $\cL_V$ at $[E]$ is, by definition, the determinant of
    cohomology
    \[
        \cL_V|_{[E]} \;\cong\; \det R\Gamma\bigl(C, E \otimes V\bigr)
        \;\cong\; \bigotimes_{i=0}^1 \bigl(\det H^i(C,E\otimes V)\bigr)^{(-1)^i}.
    \]
    Since $E = \bigoplus_j E_j^{\oplus m_j}$, we have
    \[
        R\Gamma(C,E\otimes V)
        \;\cong\; \bigoplus_{j=1}^n R\Gamma(C,E_j\otimes V)^{\oplus m_j},
    \]
    and an element $(g_1,\ldots,g_n) \in \Aut(E)$ acts diagonally on this
    direct sum.  Thus its action on the fiber $\cL_V|_{[E]}$ is given by
    multiplication by
    \[
        \prod_{j=1}^n \det(g_j)^{\sum_i (-1)^i \dim H^i(C,E_j\otimes V)}
        \;=\; \prod_{j=1}^n \det(g_j)^{\chi(C,E_j\otimes V)}.
    \]
    By construction, each $E_j$ has slope $d/r$, and by assumption we have
    $\chi(C,E_j\otimes V) = 0$ for all $j$, hence the above product is
    equal to $1$, and the stabilizer $\Aut(E)$ acts trivially on the
    fiber $\cL_V|_{[E]}$.

    We have thus shown that the line bundle $\cL_V$ on
    $\cM^{ss,\alpha}_C(r,\cL)$ has trivial stabilizer action at all
    closed points.  By Alper's descent theorem for vector bundles on
    good moduli spaces \cite[Theorem~10.3]{Alper_13}, vector bundles on
    $M^{ss,\alpha}_C(r,\cL)$ are equivalent to vector bundles on
    $\cM^{ss,\alpha}_C(r,\cL)$ with trivial stabilizer action at closed
    points.  Applied to the rank~$1$ vector bundle $\cL_V$, this yields a
    line bundle $\Lambda$ on $M^{ss,\alpha}_C(r,\cL)$ such that
    $\cL_V \cong \pi^{*}\Lambda$, where
    $\pi:\cM^{ss,\alpha}_C(r,\cL)\to M^{ss,\alpha}_C(r,\cL)$ is the good
    moduli space morphism.

    Finally, over an algebraic closure $\bar{k}$, the moduli stack
    $\cM^{ss,\alpha}_C(r,\cL)_{\bar{k}}$ identifies with an untwisted
    moduli stack of semistable vector bundles with fixed determinant, and
    the descended determinantal line bundle on 
    $M^{ss}_C(r,\cL\otimes\cN)_{\bar{k}}$ is semi-ample by
    \cite[Proposition~5.2]{Alper_Belmans_Bragg_Liang_Tajakka_22}.
    Semi-ampleness is a geometric property, so $\Lambda$ is semi-ample on
    $M^{ss,\alpha}_C(r,\cL)$.
\end{proof}

Let $f\colon C\to k$ be a smooth projective geometrically connected curve admitting a rational point $p\in C(k)$. The Leray spectral sequence for $f$ with coefficients in $\GG_m$ (similarly for $\mu_n$) yields the extended five-term exact sequence
\[
    0\to \Pic(C)\to \Pic_{C/k}(k)\to \Br(k)\to \ker\!\left(\Br(C)\to H^0(k,R^2 f_*\GG_m)\right)\xrightarrow{p_1}H^1(k,\Pic_{C/k})\to H^3(k,\GG_m).
\]

Since $p\in C(k)$ gives a section of $f$, the map $\Br(C)\to\Br(k)$ splits, and we obtain canonical decompositions
\[
    H^2(C,\GG_m)=H^2(k,\GG_m)\oplus H^1(k,\Pic_{C/k}),
\]
\[
    H^2(C,\mu_n)=H^2(k,\mu_n)\oplus H^1(k,\Pic_{C/k}[n])\oplus \ZZ/n\ZZ.
\]
Here the summand $\ZZ/n\ZZ=H^0(k,R^2 f_*\mu_n)$ is generated by the class of the $n$th-root gerbe $[\cO(p)]^{1/n}$, and measures the degree of the essentially trivial $\mu_n$-gerbes.

For $\alpha\in H^2(C,\mu_n)$, let $\beta:=p_1(\alpha)\in H^1(k,\Pic_{C/k}[n])$. Since tensoring with an $n$-torsion line bundle induces an inclusion
\[
    \Pic_{C/k}[n]\hookrightarrow \Aut\!\left(M^{ss}_{C_{\bar{k}}/\bar{k}}(n,\cO(p))\right),
\]
the class $\beta$ determines an element of $H^1\!\left(k,\Aut(M^{ss}_{C/k}(n,\cO(p)))\right)$ and therefore a twisted form of the moduli space $M^{ss}_{C/k}(n,\cO(p))$.

\begin{remark}
    Since \(\cM^{s}_{C/k}(n,\cL)\to M^{s}_{C/k}(n,\cL) \) is a $\mu_n$-gerbe, the obstruction for a point of the coarse moduli space to lift to the stack lies in the Brauer group $\Br(k)$. In particular, the Brauer obstruction to the existence of a tautological bundle on $M^{s}_{C/k}(n,\cL)$ is always the pullback of a class from $\Br(k)$. Consequently, a universal (or tautological) family need not exist on $M^{s}_{C/k}(n,\cL)$.
\end{remark}

We summarize this correspondence below.

\begin{proposition}\label{Gtwist}\textup{\cite[Proposition~3.2.2.6]{Lieblich_08}}
    Let $\alpha\in H^2(C,\mu_n)$, and write $\alpha_{\bar{k}} = a \in \ZZ/n\ZZ = H^0(k,R^2 f_* \mu_n)$ for its geometric component. Fix a line bundle $\cL\in \Pic(C)$, and let $\beta := p_1(\alpha)\in H^1(k,\Pic_{C/k}[n])$. Then the moduli space $M^{ss,\alpha}_{C/k}(n,\cL)$ is the Galois twist of the untwisted moduli space $M^{ss}_{C/k}(n,\cL(-ap))$ associated to the class $\beta$ in $H^1(k,\Aut(M^{ss}_{C/k}(n,\cL(-ap))))$.
\end{proposition}

\begin{remark}
    If $\beta$ is trivial, then the moduli space $M^{ss,\alpha}_{C/k}(n,\cO_C)$ is untwisted and isomorphic to $M^{ss}_{C/k}(n,\cO_C(-ap))$, where $a \in H^0(k,R^2 f_* \mu_n) \cong \ZZ/n\ZZ$ is the image of $\alpha$.
    
    Notice that in the case where $C$ does not admit a rational point, we no longer have a splitting of $H^2(C,\mu_n)$ as before, but the map $p_1$ still exists.  In particular, we can still identify $M^{ss,\alpha}_{C/k}(n,\cL)$ as a Galois twist of $M^{ss}_{C/k}(n,\cN)$ for some line bundle $\cN$ on $C$, not necessarily of the form $\cO_C(-ap)$.
\end{remark}

\begin{remark}\label{twistratpt}
    The above discussion generalizes to Picard schemes. Let $\alpha \in H^2(C,\mu_n)$, with $n$ invertible in $\operatorname{char}(k)$, and let $\alpha_{\bar{k}} = a \in \mathbb{Z}/n\mathbb{Z}$ denote its geometric invariant. Then one can show that the degree $i$ $\alpha$–twisted Picard scheme $\Pic^{\,i}_{C/k,\alpha}$ is obtained as the Galois twist of $\Pic^{\,i-a}_{C/k}$ by the cocycle in $H^1\!\left(k,\Pic_{C/k}[n]\right)$ corresponding to~$\alpha$.  (Equivalently, it is the torsor obtained by translating $\Pic_{C/k}$ by the $n$–torsion class determined by~$\alpha$.)

    Consequently, whenever the image of $\alpha$ in $H^1(k,\Pic_{C/k}[n])$ is nontrivial, the scheme $\Pic^{\,i}_{C/k,\alpha}$ is a nontrivial $\Pic^0_{C/k}$–torsor.  In particular, it has no $k$–rational point.

    Conversely, if $\alpha$ is essentially trivial, then $\Pic^{\,i}_{C/k,\alpha}$ is the trivial $\Pic^0_{C/k}$–torsor.  Hence it is (noncanonically) isomorphic to $\Pic^0_{C/k}$, and in particular contains a $k$–rational point corresponding to an $\alpha$–twisted line bundle on~$C$.

    We refer the reader to \cite{Huybrecht_Mattei_25} for further discussion of this construction.
\end{remark}

\subsection{Bundles on gerbes} In this subsection, we study line bundles and higher-rank vector bundles on $\mu_n$-gerbes, illustrating the general picture through two explicit examples.

\begin{theorem}\label{classifyPic}
    Let $\alpha \in H^2(C,\mu_n)$ and let $\beta$ denote its optimal part, so that 
    $|\beta| = m$.  Then
    \[
        \Pic(\sC)\;=\;\coprod_{i\in \ZZ/(\frac{n}{m})\ZZ} \Pic_{\alpha^{\,im}}(C)\;\cong\;\coprod_{i\in \ZZ/(\frac{n}{m})\ZZ}\Pic(C).
    \]
    In particular, $\Pic(\sC)\cong \Pic_{\sC/k}(k)$ if and only if $\alpha \ \text{is essentially trivial}.$
\end{theorem}

\begin{proof}
    Since $|\beta|=m$, the class $\alpha^{m}$ maps to zero in $\Br(C)$, and is therefore essentially trivial.  Consequently, $\Pic_{\alpha^{k}}(C)$ is nonempty if and only if $m \mid k$, which yields
    \[
        \Pic(\sC)=\coprod_{i\in \ZZ/(n/m)\ZZ} \Pic_{\alpha^{\,im}}(C).
    \]

    For each $i$, essential triviality of $\alpha^{m}$ implies the existence of an $\alpha^{\,im}$-twisted line bundle of some degree. Tensoring with the inverse of such a bundle provides a noncanonical isomorphism $\Pic_{\alpha^{\,im}}(C)\;\xrightarrow{\;\sim\;}\Pic(C),$ giving the second isomorphism in the statement.

    Finally, by Remark~\ref{twistratpt}, the twisted Picard scheme $\Pic_{C/k,\alpha^j}^{i}$ admits a $k$-point for every $i$ and $j$ if and only if $\alpha$ is essentially trivial. This establishes the last claim.
\end{proof}

Next, we recall that Nironi \cite{Nironi_09} constructed moduli spaces of semistable sheaves on Deligne--Mumford stacks by adapting the GIT construction via Quot schemes. His theory applies in particular to decorated sheaves, including twisted sheaves and parabolic bundles. For our purposes, we record the following result.

\begin{theorem}\label{nironi}\textup{\cite[Proposition~7.3]{Nironi_09}}
    Let $\pi\colon \sX\to X$ be a $G$-banded gerbe over $X$, where $k$ is an algebraically closed field, $X$ is a smooth projective variety and $G$ is either $\GG_m$ or $\mu_n$. Let $\chi=(\chi_1,\dots,\chi_n)$ be characters of $G$, and let $\cP_\chi=(P_{\chi_1},\dots,P_{\chi_n})$ denote the set of $n$-tuples of Hilbert polynomials labelled by~$\chi$ with $\sum_{i=1}^n P_{\chi_i}\;=\;P .$ Then the moduli stack and the good moduli space of torsion-free semistable sheaves on $\sX$ decompose as
    \[
        \cM^{ss}(P)\;\cong\;\coprod_{P_\chi\in \cP}\ \coprod_{i=1}^n \cM^{ss}_{\chi}(P_{\chi_i}),\qquad M^{ss}(P)\;\cong\;\coprod_{P_\chi\in \cP}\ \coprod_{i=1}^n M^{ss}_{\chi}(P_{\chi_i}).
    \]
\end{theorem}

Note that although \cite{Nironi_09} assumes the base field \(k\) to be algebraically closed, the proofs apply verbatim under the standing hypotheses fixed at the beginning of this subsection.

\begin{definition}\label{mix}
    Let \(\cF \in \Coh(X,\alpha^{i})\) and \(\cG \in \Coh(X,\alpha^{j})\) be \(\sX\)-coherent sheaves with \(i\neq j\). We call the direct sum \(\cF \oplus \cG\) a \emph{mixed} \(\sX\)-coherent sheaf. We denote the moduli space of mixed \(\sX\)-coherent sheaves by \(\operatorname{Mix}(X)\), and the moduli space of mixed \(\sX\)-vector bundles of rank \(r\) and determinant \(\cL\) by \(\operatorname{Mix}(X,r,\cL)\).
\end{definition}

\begin{proposition}\label{vbdecomp}
    Let $\sX \to X$ be a $\mu_n$-gerbe with class $\alpha\in H^2(X,\mu_n)$. Then the moduli stack of vector bundles of rank $r$ and determinant $\cL$ on $\sX$ admits a disjoint decomposition
    \[
        \cM_{\sX}(r,\cL)\;=\;\coprod_{i\in \ZZ/n\ZZ} \cM^{\alpha^{i}}_X(r,\cL)\;\sqcup\;\Mix(X,r,\cL).
    \]
\end{proposition}

\begin{proof}
    By the character decomposition for $\mu_n$-gerbes (Theorem~\ref{decomp}), any vector bundle $V$ on $\sX$ decomposes canonically as $V \;=\; \bigoplus_{i\in\ZZ/n\ZZ} V_i,$ where each $V_i$ is an $\alpha^{i}$-twisted vector bundle on $X$. Thus either $V$ lies entirely in a single character summand (giving the components $\cM^{\alpha^{i}}_X(r,\cL)$), or it has nonzero contributions from at least two summands, in which case it is a mixed bundle in the sense of Definition~\ref{mix}. This yields the stated decomposition.
\end{proof}

\begin{remark}
    By Theorem~\ref{classifyPic}, the determinant line bundle $\cL$ of a vector bundle on the gerbe $\sX$ must necessarily be an $\alpha^{\,j}$\!-twisted line bundle for some $0 \le j < n$.
\end{remark}

\begin{lemma}\label{mixvb}
    Let $\sX \to X$ be a $\mu_n$–gerbe represented by $\alpha \in H^{2}(X,\mu_n)$, and let $\cF \in \Mix^{ss}(X,r,\cL)$ be a mixed semistable $\sX$–vector bundle, where $\cL$ is an $\alpha^{j}$–twisted line bundle on $X$. Then $\cF$ decomposes as
    \[
        \cF \;\cong\;\cF_{0} \oplus \cF_{\alpha} \oplus\cdots \oplus\cF_{\alpha^{n-1}},
    \]
    where each $\cF_{\alpha^{i}}$ is an $\alpha^{i}$–twisted vector bundle of rank $r_i$. These summands satisfy:
    \begin{enumerate}
        \item[(i)] $\displaystyle \sum_{i=0}^{n-1} r_i = r$.
        
        \item[(ii)] For every $i$ with $r_i \neq 0$, one has  
        \[
            \frac{\deg \cF_0}{r_0}\;=\;\frac{\deg \cF_{\alpha}}{r_1}\;=\;\cdots\;=\;\frac{\deg \cF_{\alpha^{n-1}}}{r_{n-1}},
        \]
        so all nonzero summands have the same slope.

        \item[(iii)] $\displaystyle \sum_{i=0}^{n-1} \deg \cF_{\alpha^{i}}= \deg \cL.$
        
        \item[(iv)] The twisting indices satisfy  
        \[
            \sum_{i=0}^{n-1} r_i\, i \equiv j \pmod{n}.
        \]
    \end{enumerate}
\end{lemma}

\begin{proof}
    (i) This follows immediately by computing the ranks of the direct summands in the decomposition $\cF = \bigoplus_{i=0}^{n-1} \cF_{\alpha^{i}}$.

    \medskip
    (ii) If $\cF$ and $\cG$ are vector bundles on $\sX$, then $\cF \oplus \cG$ is semistable if and only if $\mu(\cF)=\mu(\cG)$ and both $\cF$ and $\cG$ are semistable. Applying this to the decomposition of $\cF$, we obtain the desired equality of slopes for all summands with $r_i\neq 0$.

    \medskip
    (iii) Since $\det : K(\sX) \to \Pic(\sX)$ is a group homomorphism, we have
    \[
        \det(\cF)\;=\;\bigotimes_{i=0}^{n-1} \det(\cF_{\alpha^{i}})\;=\;\cL.
    \]
    Taking degrees gives $\deg(\cL) = \sum_{i=0}^{n-1} \deg(\cF_{\alpha^{i}})$.

    \medskip
    (iv) As in part (iii), the determinant of $\cF$ is an $\Bigl(\prod_{i=0}^{n-1} \alpha^{r_i i}\Bigr)$–twisted line bundle. Since $\det(\cF)=\cL$ is $\alpha^{j}$–twisted, we must have
    \[
        \prod_{i=0}^{n-1} \alpha^{r_i i} = \alpha^{j},
    \]
    which is equivalent to the congruence $\sum_{i=0}^{n-1} r_i i \equiv j \pmod{n}$.
\end{proof}

We classify two examples in the following:

\begin{theorem}\label{rank2}
    Let $\sX \to X$ be a $\mu_{2}$-gerbe representing $\alpha \in H^{2}(X,\mu_{2})$ with $\alpha_{\bar{k}} = 0$. Then
    \[
        M^{ss}_{\sX}(2,\cO_X)\;\cong\;M^{ss}_{X}(2,\cO_X)\;\coprod\;M^{ss,\alpha}_{X}(2,\cO_X).
    \]
\end{theorem}

\begin{proof}
    By Proposition~\ref{vbdecomp}, it suffices to show that $\Mix(X,2,\cO_X)=\varnothing$. By Lemma~\ref{mixvb}(i), the only remaining possibility for a mixed rank-\(2\) vector bundle is one of the form $\cL \oplus \cL_{\alpha}$, where $\cL \in \Pic(X)$ and $\cL_{\alpha} \in \Pic_{\alpha}(X)$.

    However,
    \[
        \det(\cL \oplus \cL_{\alpha})= \det(\cL)\otimes \det(\cL_{\alpha})= \cL \otimes \cL_{\alpha},
    \]
    which is an $\alpha$-twisted line bundle. This contradicts Lemma~\ref{mixvb}(iv), which forces the determinant of a mixed bundle of type $(1,1)$ to be \emph{untwisted}. Hence no such mixed semistable bundle exists, and the claimed decomposition follows.
\end{proof}

\begin{theorem}
    Let $\sX \to X$ be a $\mu_{3}$\nobreakdash-gerbe representing $\alpha \in H^{2}(X,\mu_{3})$ with $\alpha_{\bar{k}} = 0$.  Then
    \[
        M^{ss}_{\sX}(3,\cO_{X})\;\cong\;M^{ss}_{X}(3,\cO_{X})\,\coprod\,M^{ss,\alpha}_{X}(3,\cO_{X})\,\coprod\,M^{ss,\alpha^{2}}_{X}(3,\cO_{X})\,\coprod\, \Mix(X,3,\cO_{X}),
    \]
    where $\Mix(X,3,\cO_{X})$ is the subspace of
    \[
        \Pic^{0}_{X} \;\oplus\; \Pic^{0}_{X,\alpha} \;\oplus\;\Pic^{0}_{X,\alpha^{2}}
    \]
    consisting of closed points $(\cL,\cL_{\alpha},\cL_{\alpha^{2}})$ satisfying
    \[
        \cL \otimes \cL_{\alpha} \otimes \cL_{\alpha^{2}}\;\cong\;\cO_{X},
    \]
    after a finite separable extension of the base field.
\end{theorem}

\begin{proof}
    By Proposition~\ref{vbdecomp}, it suffices to describe $\Mix(X,3,\cO_{X})$. By Lemma~\ref{mixvb} (i), the possible rank partitions are
    \[
        (2,1,0),\ (1,2,0),\ (0,1,2),\ (0,2,1),\ (2,0,1),\ (1,0,2),\ (1,1,1).
    \]
    Lemma~\ref{mixvb}(iv) then forces \(\sum_{i=0}^{2} r_{i} i \equiv 0 \pmod 3\), which leaves only the case $(1,1,1)$.

    Next, Lemma~\ref{mixvb}(ii) and (iii) imply that each summand must have degree $0$. Thus a mixed bundle is given by a triple
    \[
        (\cL,\,\cL_{\alpha},\,\cL_{\alpha^{2}}) \in \Pic^{0}_{X} \oplus \Pic^{0}_{X,\alpha} \oplus \Pic^{0}_{X,\alpha^{2}}
    \]
    satisfying the determinant condition
    \[
        \cL \otimes \cL_{\alpha} \otimes \cL_{\alpha^{2}} \cong \cO_{X},
    \]
    which is exactly the locus stated in the theorem.
\end{proof}

\begin{remark}
    In general, for a $\mu_{n}$\nobreakdash-gerbe $\sX \to X$, the moduli space $M_{\sX}^{ss}(r,\cO_{X})$ decomposes as the disjoint union of the components $M_{X}^{ss,\alpha^{i}}(r,\cL_{i}),$ for $0 \le i < n,$ together with the locus of mixed vector bundles. The mixed part is itself a disjoint union of products of moduli spaces of $\alpha^{i}$\nobreakdash-twisted semistable vector bundles of smaller ranks and fixed determinants.
\end{remark}

\section{Moduli space of twisted vector bundles of rank~2 with trivial determinant}In this section, we explicitly describe the moduli spaces of $\alpha$–twisted  vector bundles of rank~2 and trivial determinant. We begin with the essentially trivial case, where the situation parallels the classical moduli space of vector bundles on a curve. We then treat the optimal case, in which the obstruction given by the Brauer class $\alpha$ prevents the determinantal line bundle from descending from the algebraic closure. In this latter setting, we show how the Brauer class controls both the geometry of the moduli space and its deviation from the untwisted case.

\subsection{Essentially trivial case} In this subsection, our arguments follow those of \cite{Beauville_Narasimhan_Ramanan_89} and \cite{Narasimhan_Ramanan_69}, with only minor modifications required in the twisted setting. Unless otherwise specified, we fix a smooth, projective, and geometrically connected curve $C$ of genus $2$ over $k$.

\begin{construction}\label{theta}
    (Theta divisor) This construction works over an arbitrary field $k$. Let $\Delta : C \to C \times C$ be the diagonal morphism, $x \mapsto (x,x)$. The associated divisor $\Delta \subset C \times C$ determines a line bundle $\cL_\Delta := \cO_{C \times C}(\Delta).$

    For each closed point $x \in C$, with residue field $k(x)$, the fibre $C \times \{x\}$ is canonically isomorphic to $C_{k(x)}$, and the restriction $\cL_x := \cL_\Delta|_{C \times \{x\}}$ is a line bundle of degree $1$ on $C_{k(x)}$. Thus $\cL_\Delta$ defines a family of degree-$1$ line bundles on $C$ parametrized by $C$.
    
    By the universal property of the Picard scheme, any family $\{\cL_t\}$ of degree-$d$ line bundles parametrized by an algebraic space $T$ determines a morphism $T \to \Pic^d_{C/k}$. Applying this to the family $\{\cL_x\}_{x \in C}$ yields a morphism $C \rightarrow \Pic^1_{C/k}.$

    Since $\dim \Pic^1_{C/k} = g = 2$, the image of this morphism is a nonsingular divisor on $\Pic^1_{C/k}$, denoted by $\theta$. We denote the associated line bundle by $\cL_\theta := \cO_{\Pic^1_{C/k}}(\theta)$.
\end{construction}

Let $\alpha \in H^2(C,\mu_2)$ be an essentially trivial $\mu_2$-gerbe. By Proposition~\ref{esstrivdeg}, there exists an $\alpha$-twisted line bundle $\cL_\alpha$ of degree $k_\alpha := \tfrac{1}{2}\,\varphi(\alpha)$ such that $(\cL_\alpha)^{\otimes 2} \cong \cL$ for some line bundle $\cL$ on $C$ of degree $\varphi(\alpha)$.

Tensoring by $\cL_\alpha^\vee$ gives isomorphisms between the ordinary and twisted Picard schemes:
\[
    \tw_0 : \Pic^0_{C/k} \xrightarrow{\;\simeq\;} \Pic^{-k_\alpha}_\alpha,\qquad \tw_1 : \Pic^1_{C/k} \xrightarrow{\;\simeq\;} \Pic^{\,1-k_\alpha}_\alpha.
\]

Applying $\tw_1$ to the map $t : C \rightarrow \Pic^1_{C/k}$ gives a morphism
\[
    t_\alpha := \tw_1 \circ t : C \longrightarrow \Pic^{\,1-k_\alpha}_\alpha.
\]
Concretely, $t_\alpha(x)$ corresponds to the $\alpha$-twisted line bundle $\cL_x \otimes \cL_\alpha^\vee$ on $C$.

By Construction~\ref{theta}, this family of twisted line bundles defines a divisor
\[
    \theta_\alpha := t_\alpha(C) \subset \Pic^{\,1-k_\alpha}_\alpha(C),
\]
called the \emph{twisted theta divisor}. We denote the associated line bundle by $\cL_{\theta_\alpha}$. By construction, one has $\tw_1^{*}\cL_{\theta_\alpha} = \cL_\theta.$

In the following, we assume that $C(k)\neq\varnothing$. In this case, the component $\Pic^1_{C/k}$ of the Picard scheme is a torsor under $\Pic^0_{C/k}$ that is trivialized by the choice of a $k$-rational point on $C$. Thus we do not distinguish between the Picard scheme $\Pic^1_{C/k}$ and the notation $\Pic^1(C)$ for degree $1$ line bundles on $C$.

\begin{theorem}
    Let $\alpha$ be an essentially trivial $\mu_2$-gerbe with $\alpha_{\bar{k}} = 0$. Then the morphism
    \[
        M^{ss,\alpha}_C(2,\cO_C)\;\longrightarrow\;\PP\!\bigl(H^0\bigl(\Pic^1(C),\, \cL_{\theta}^{\otimes 2}\bigr)\bigr),
    \]
    induced by the linear system $|\cL_\theta^{\otimes 2}|$, is an isomorphism.
\end{theorem}

\begin{proof}[Sketch of proof]
    This follows the argument of \cite{Narasimhan_Ramanan_69}, with the only changes occurring in the use of twisted line bundles. We outline the main steps.

    \medskip
    \noindent\textbf{Step 1.}
    Every $\alpha$-twisted semistable vector bundle $V \in M_C^{ss,\alpha}(2,\cO_C)$ is $S$-equivalent to an extension
    \[
        0 \longrightarrow \xi^{-1} \longrightarrow V \longrightarrow \xi \longrightarrow 0
    \]
    for some $\xi \in \Pic^1_{\alpha}(C)$. Thus the $S$-equivalence class of $V$ is determined by an element of
    \[
        \Ext^1(\xi, \xi^{-1}) \;\cong\; H^1(C,\xi^{-2}).
    \]

    \medskip
    \noindent\textbf{Step 2.}
    Fix $\xi \in \Pic^1_{\alpha}(C)$. There is an exact sequence on $\Pic^1_{\alpha}(C)$
    \[
        0 \longrightarrow E \longrightarrow \Pic^1_{\alpha}(C)\times H^0\!\left(\Pic^1_{\alpha}(C), \cL_{\theta_\alpha}^{\otimes 2}\right)\xrightarrow{\ \mathrm{ev}\ }\cL_{\theta_\alpha}^{\otimes 2}\longrightarrow 0,
    \]
    where $\mathrm{ev}$ is the evaluation map. The fiber of $E$ over $\xi$ is identified with $H^1(C,\xi^{-2})$. In particular, $\PP\bigl(H^1(C,\xi^{-2})\bigr)$ parametrizes the set of effective divisors on $\Pic^1_{\alpha}(C)$ passing through $\xi$ and linearly equivalent to $2\theta_\alpha$.

    \medskip
    \noindent\textbf{Step 3.}
    For a semistable $\alpha$-twisted bundle $V$ of rank $2$ and trivial determinant, define
    \[
        C_V \;=\;\{\ \xi \in \Pic^1_{\alpha}(C) \mid H^0(C,\xi \otimes V) \neq 0 \ \}.
    \]
    Then $C_V$ is the support of a uniquely determined effective divisor $D_V$ on $\Pic^1_{\alpha}(C)$, linearly equivalent to $2\theta_\alpha$. Moreover, $D_V$ depends only on the $S$-equivalence class of $V$. This assigns to $V$ a well-defined point $D_V \;\in\;\PP\!\left(H^0\bigl(\Pic^1_{\alpha}(C),\cL_{\theta_\alpha}^{\otimes 2}\bigr)\right)$.
    
    \medskip
    \noindent\textbf{Conclusion.}
    The map
    \[
        D : M_C^{ss,\alpha}(2,\cO_C)\longrightarrow\PP\!\left(H^0\bigl(\Pic^1_{\alpha}(C),\cL_{\theta_\alpha}^{\otimes 2}\bigr)\right),\qquad V \longmapsto D_V,
    \]
    is bijective on geometric points and compatible with families. Its inverse is given by recovering the extension class $H^1(C,\xi^{-2})$ from the divisor passing through~$\xi$. Thus $D$ is an isomorphism.
\end{proof}

\begin{remark}
    Arguing as above, one can similarly show that if $\cL$ is a line bundle of odd degree and $\alpha_{\bar{k}} = 1 \in H^2(C_{\bar{k}},\mu_2)$, then there is an isomorphism
    \[
    M^{ss,\alpha}_C(2,\cL)\;\xrightarrow{\;\sim\;}\PP\!\bigl(H^0(\Pic^1_\alpha(C), \cL_{\theta_\alpha}^{\otimes 2})\bigr).
    \]
\end{remark}

\subsection{Generalized theta divisor and higher genus cases}
In this subsection, we recover a result of \cite{Beauville_Narasimhan_Ramanan_89} for smooth projective curves $C$ of genus $\geq 2$, in the setting of an essentially trivial $\mu_n$-gerbe. To avoid any ambiguity, we assume throughout that $\operatorname{char}(k)=0$. However, we note that in the special case of rank $2$ and trivial determinant when $g=2$, the statement remains valid whenever $\operatorname{char}(k)\neq 2$.

\begin{definition}
    Let $V$ be an $\alpha$-twisted vector bundle on $C$. The \emph{generalized twisted theta divisor} associated to $V$ is the set
    \[
        \Theta_V\;:=\;\bigl\{\, [E] \in M^{ss,\alpha}_{C/k}(n,\cL)\;\big|\;H^0\!\bigl(C,\, E \otimes V^{\vee}\bigr) \neq 0\bigr\}.
    \]
\end{definition}

\begin{remark}
    In the classical (untwisted) setting, the locus $\Theta_V$ is a Cartier divisor on $M^{ss}_C(n,\cL)$, and its associated line bundle is precisely the corresponding determinantal line bundle.  In particular, $\Theta_V$ generates the Picard group of the moduli space of vector bundles, as shown in \cite{Drezet_Narasimhan_89}.
\end{remark}

\begin{proposition}\label{pulltheta}
    Let $W$ be an $\alpha$-twisted vector bundle on $C$ of rank $n$ with $\det(W)=\cL\otimes\cN$, and let $V$ be an ordinary vector bundle of rank $n$ with determinant $\cL$. Assume $\alpha\in H^2(C,\mu_n)$ is essentially trivial, and let $\cL_\alpha$ be the invertible $\alpha$-twisted line bundle provided by Lemma~\ref{esstriv}. Then the isomorphism
    \[
        \tw := (-)\otimes \cL_\alpha \colon M^{ss}_C(n,\cL)\xrightarrow{\;\sim\;}M^{ss,\alpha}_C(n,\cL\otimes\cN)
    \]
    satisfies $\tw^*(\cL_W) \;\cong\; \cL_V .$
\end{proposition}

\begin{proof}
    Since $\alpha$ is essentially trivial, tensoring with $\cL_\alpha$ induces an equivalence between untwisted and $\alpha$-twisted bundles of rank $n$.  Under this equivalence, the twisted bundle $W$ corresponds to the ordinary bundle $W\otimes \cL_\alpha^\vee$.
    
    By construction of determinantal line bundles and of $\tw$, we have
    \[
        \tw^*(\cL_W)=\cL_{\,W\otimes\cL_\alpha^\vee}.
    \]
    By Proposition~\ref{detline}, the determinantal line bundle of a rank $n$ bundle depends only on its rank and determinant.  As $\operatorname{rk}(V)=\operatorname{rk}(W\otimes\cL_\alpha^\vee)$ and $\det(V)\cong\det(W\otimes\cL_\alpha^\vee)$, we obtain
    \[
        \cL_{\,W\otimes\cL_\alpha^\vee}\cong \cL_V.
    \]
    Combining the two equalities yields $\tw^*(\cL_W)\cong \cL_V$, as claimed.
\end{proof}

\begin{remark}
    If $V$ is an $\alpha$-twisted vector bundle, then $\Theta_V$ is the vanishing locus of the section of the determinantal line bundle $\cL_V$.  In particular, $\cL_V \;\cong\; \cO(\Theta_V)$. By Proposition~\ref{detline}, the line bundle $\cO(\Theta_V)$ depends only on the rank and determinant of $V$, and hence is independent of the choice of $V$.  
    
    In what follows, for an essentially trivial $\mu_n$-gerbe $\alpha$, we denote by $\Theta_\alpha$ the generator of the Picard group $\Pic\bigl(M^{ss,\alpha}_C(n,\cL)\bigr)$ corresponding to this twisted theta divisor.
\end{remark}

Let $V$ be an $\alpha$-twisted vector bundle on $C$.  Define
\[
    D_V \;:=\; \bigl\{\, \xi \in \Pic^{g-1}_{\alpha}(C) \;\big|\; H^0\bigl(C, \Hom(\xi^{-1},V)\bigr)\neq 0 \,\bigr\}.
\]
Then the assignment
\[
    D \colon M^{ss,\alpha}_C(n,\cL) \dashrightarrow \bigl|\,n\theta_{\alpha}\,\bigr|,\qquad V \longmapsto D_V,
\]
defines a rational map from the moduli space $M^{ss,\alpha}_C(n,\cL)$ into the linear system $|n\theta_{\alpha}|$ on $\Pic^{g-1}_{\alpha}(C)$, in direct analogy with the construction of \cite{Beauville_Narasimhan_Ramanan_89}.

\begin{theorem}\label{bnr}
    Assume that $\alpha$ is an essentially trivial $\mu_n$-gerbe. Then the assignment
    \[
        D \colon M^{ss,\alpha}_C(n,\cO_C) \dashrightarrow \Pic^{\,g-1-k_\alpha}_\alpha(C),\qquad V \longmapsto D_V,
    \]
    induces an isomorphism
    \[
        H^0\!\left(\Pic^{\,g-1-k_\alpha}_\alpha(C),\cO(n\theta_\alpha)\right)\;\xrightarrow{\;\sim\;}H^0\!\left(M^{ss,\alpha}_C(n,\cO_C),\Theta_\alpha\right)^{\!\vee},
    \]
    and the rational map $D$ is the morphism defined by the complete linear system $\bigl|\,n\theta_\alpha\,\bigr|$.
    
    In particular,
    \[
        \dim_k H^0\!\left(M^{ss,\alpha}_C(n,\cO_C),\Theta_\alpha\right)= n^g .
    \]
\end{theorem}

\begin{proof}
    Since $\alpha$ is essentially trivial, tensoring with the invertible $\alpha$-twisted line bundle $\cL_\alpha$ induces isomorphisms
    \[
        \tw_{g-1} \colon \Pic^{\,g-1}(C) \xrightarrow{\;\sim\;} \Pic^{\,g-1-k_\alpha}_\alpha(C),\qquad\tw \colon M^{ss}_C(n,\cO_C) \xrightarrow{\;\sim\;} M^{ss,\alpha}_C(n,\cO_C),
    \]
    and these carry the classical theta divisor $\theta$ and the determinant line bundle $\Theta$ to their twisted counterparts:
    \[
        \tw_{g-1}^*(\theta_\alpha)=\theta, \qquad\tw^*(\Theta_\alpha)=\Theta.
    \]

    Thus both $\Pic^{\,g-1-k_\alpha}_\alpha(C)$ together with $\cO(n\theta_\alpha)$ and $M^{ss,\alpha}_C(n,\cO_C)$ together with $\Theta_\alpha$ are obtained from the classical objects by pullback along these isomorphisms.
    
    The construction defining $D_V$ is likewise transported from the untwisted case via $\tw$ and $\tw_{g-1}$.  Hence the map
    \[
        M^{ss}_C(n,\cO_C)\;\dashrightarrow\; \Pic^{\,g-1}(C)
    \]
    corresponds, under $\tw$ and $\tw_{g-1}$, to the twisted map
    \[
        D\colon M^{ss,\alpha}_C(n,\cO_C)\dashrightarrow \Pic^{\,g-1-k_\alpha}_\alpha(C).
    \]
    
    By \cite[Theorem~3]{Beauville_Narasimhan_Ramanan_89}, the classical map induces an isomorphism
    \[
        H^0\!\left(\Pic^{\,g-1}(C), \cO(n\theta)\right) \xrightarrow{\;\sim\;}H^0\!\left(M^{ss}_C(n,\cO_C),\Theta\right)^{\!\vee}.
    \]
    and this remains an isomorphism after descent from $\bar{k}$ to $k$. Pulling back this isomorphism by $\tw$ and $\tw_{g-1}$ yields precisely the twisted isomorphism claimed in the statement.
    
    Finally, the dimension formula $h^0(M^{ss,\alpha}_C(n,\cO_C),\Theta_\alpha)=n^g$ follows from the classical computation $h^0(M^{ss}_C(n,\cO_C),\Theta)=n^g$.
\end{proof}

We finally record a theorem from the literature describing the geometric structure of the linear system appearing in Theorem~\ref{bnr}.

\begin{theorem}\textup{\cite{Beauville_94}}\label{geometry}
    Let $C$ be a smooth projective curve of genus $g \ge 2$ over $k$, and let $\alpha \in H^{2}(C,\mu_{2})$ be an essentially trivial gerbe. For the morphism of Theorem~\ref{bnr}
    \[
        \varphi \colon M^{ss,\alpha}_{C/k}(2,\cO_C)\;\longrightarrow\;|\Theta_{\alpha}|,
    \]
    the following hold:
    \begin{enumerate}
        \item[(i)] If $g=2$, then 
        \[
            M^{ss,\alpha}_{C/k}(2,\cO_C)\;\cong\;|\Theta_{\alpha}|\;\cong\;\PP^{3}.
        \]
        \item[(ii)] If $g \ge 3$ and $C$ is hyperelliptic, then $\varphi$ is a finite morphism of degree~$2$ onto its image.
        \item[(iii)] If $g \ge 3$ and $C$ is nonhyperelliptic, then $\varphi$ is birational onto its image.
        \item[(iv)] Moreover, if $g=3$, or if $C$ is generic in moduli, then $\varphi$ is an embedding.
    \end{enumerate}
\end{theorem}

\begin{remark}
    The proof is identical to the untwisted case treated in \cite{Beauville_94}, because essentially trivial gerbes differ from the classical setting only by tensoring with a fixed invertible twisted line bundle, which does not affect the linear system $|2\theta|$ nor the geometry of the map~$\varphi$.
\end{remark}

\subsection{The optimal case via obstruction} In this section, we explain how classes in \(H^{1}\!\left(k,\Pic_{C/k}[n]\right)\) appear as the obstruction to descending the determinantal line bundle from the algebraic closure \(\bar{k}\) to the base field \(k\).

\begin{lemma}\label{linear}
Let $V$ be a vector bundle on $C$, and let $\cL_V$ be the associated determinantal line bundle on $M^{ss}_{C/k}(n,\cL)$. Then there is a morphism of group sheaves
\[
    \psi\colon \Pic_{C/k}[n]\longrightarrow \mu_n
\]
characterized as follows: for any $k$-scheme $S$ and any $\cM\in \Pic_{C/k}[n](S)$, the induced automorphism
\[
    T_\cM \colon M^{ss}_{C/k}(n,\cL)_S \longrightarrow M^{ss}_{C/k}(n,\cL)_S,\qquad [E]\longmapsto[E\otimes\cM]
\]
gives rise to a scalar automorphism
\[
    \psi(\cM)\in\Aut\bigl((\cL_V)_S\bigr)\cong\GG_m(S),
\]
and this construction is compatible with base change in $S$.
\end{lemma}

\begin{proof}
For $\cM\in\Pic_{C/k}[n](S)$, pullback along $T_\cM$ gives a canonical isomorphism of line bundles
\[
    T_\cM^*(\cL_V)_S \;\xrightarrow{\;\sim\;} (\cL_{V\otimes\cM})_S.
\]
This is obtained fiberwise as follows: for any geometric point $[E]$,
\[
    (T_\cM^*\cL_V)_S|_{[E]}
    = \cL_V|_{[E\otimes\cM]}
    = \prod_{i=0}^1 \bigl(\det H^i(C,E\otimes \cM\otimes V)\bigr)^{\otimes (-1)^i}
    = \cL_{V\otimes\cM}|_{[E]}.
\]
These identify the fibers of the two line bundles, hence glue to the isomorphism above.

Since $\cM\in\Pic_{C/k}[n]$, Proposition~\ref{detline} gives a canonical isomorphism
\[
    (\cL_{V\otimes\cM})_S \;\xrightarrow{\;\sim\;} (\cL_V)_S.
\]
Composing the two maps shows that $T_\cM$ induces an automorphism 
\[
    \Phi_\cM : (\cL_V)_S \xrightarrow{\sim} (\cL_V)_S.
\]

Because $\Aut((\cL_V)_S)\cong\GG_m(S)$, the automorphism $\Phi_\cM$ is multiplication by a unique scalar, which we denote by $\psi(\cM)$. This gives a morphism
\[
    \psi_S: \Pic_{C/k}[n](S)\longrightarrow\GG_m(S),
\]
compatible with base change in $S$, hence a morphism of sheaves $\psi:\Pic_{C/k}[n]\to\GG_m$.

Finally, $T_\cM^n=\id$ implies $\Phi_\cM^{\,n}=\id$, hence $\psi(\cM)^n=1$, so the image lies in $\mu_n$.
\end{proof}

\begin{theorem}\label{main1}
    There is a group homomorphism
    \[
        H^1(k,\Pic_{C/k}[n]) \;\longrightarrow\; \Br(k)
    \]
    sending a class $\alpha$ to the obstruction to descending the determinantal line bundle.
\end{theorem}

\begin{proof}
    Let $\alpha \in H^1\bigl(k,\Pic_{C/k}[n]\bigr)$. Then $\alpha_{\bar{k}} = 0$, so over $\bar{k}$ the twisting is trivial. Since $\Pic\bigl(M^{ss,\alpha}_C(n,\cO_C)\times_k \bar{k}\bigr) \;\cong\; \ZZ,$ the Galois group $G = \Gal(\bar{k}/k)$ acts trivially on this Picard group. In particular, the determinantal line bundle $\Theta_{\alpha_{\bar{k}}}$ on $M^{ss,\alpha}_C(n,\cO_C)\times_k \bar{k}$ defines a $k$-rational point of the Picard group.
    
    We now show that $\alpha$ is exactly the Brauer obstruction to descending $\Theta_{\alpha_{\bar{k}}}$ to $k$.

    Choose a $1$--cocycle representing $\alpha$:
    \[
        \alpha \colon G \longrightarrow \Pic(C_{\bar{k}})[n],\qquad\sigma \longmapsto \cL_\sigma,
    \]
    so that for all $\sigma,\tau \in G$ we have an isomorphism
    \[
        \cL_{\sigma\tau} \;\cong\; \cL_\sigma \otimes \sigma^{*}\cL_\tau.
    \]
    Since tensoring by an $n$--torsion line bundle acts on the moduli space, Lemma~\ref{linear} gives a homomorphism
    \[
        \Pic_{C/k}[n] \longrightarrow \Aut(\Theta_{\alpha_{\bar{k}}})[n]\;\cong\; \mu_n,
    \]
    and hence a $1$--cocycle
    \[
        G \longrightarrow \Aut(\Theta_{\alpha_{\bar{k}}})\;\cong\; \GG_m,\qquad\sigma \longmapsto \cL_\sigma^{*},
    \]
    where $\cL_\sigma^{*}$ denotes the automorphism of $\Theta_{\alpha_{\bar{k}}}$ induced by tensoring with $\cL_\sigma$.

    To descend the line bundle $\Theta_{\alpha_{\bar{k}}}$ to $k$ one needs isomorphisms
    \[
        \phi_\sigma : \sigma^{*}\Theta_{\alpha_{\bar{k}}}\xrightarrow{\;\sim\;}\Theta_{\alpha_{\bar{k}}}\quad (\sigma \in G)
    \]
    satisfying the cocycle condition $\phi_{\sigma\tau} = \phi_\sigma \circ \sigma^{*}\phi_\tau$. Starting from the system of automorphisms $\{\cL_\sigma^{*}\}$, the failure of this condition is measured by the $2$--cocycle
    \[
        c(\sigma,\tau)\;:=\;\cL_\sigma^{*} \circ \sigma^{*}(\cL_\tau^{*}) \circ(\cL_{\sigma\tau}^{*})^{-1}\;\in\; \Aut(\Theta_{\alpha_{\bar{k}}})\;\cong\; \GG_m.
    \]
    By Lemma~\ref{linear}, each $c(\sigma,\tau)$ actually lies in the $n$-torsion subgroup $\mu_n \subset \GG_m$, so $\{c(\sigma,\tau)\}$ defines a class in $H^2(k,\mu_n)\subset \Br(k)$.

    This class vanishes if and only if we can modify the system $\{\cL_\sigma^{*}\}$ by scalars so as to satisfy the cocycle condition, i.e.\ if and only if $\Theta_{\alpha_{\bar{k}}}$ descends to a line bundle over $k$.  Thus the class of $\{c(\sigma,\tau)\}$ is precisely the Brauer obstruction to descending the determinantal line bundle $\Theta_{\alpha_{\bar{k}}}$.

    We therefore obtain a well-defined map
    \[
        H^1\bigl(k,\Pic_{C/k}[n]\bigr)\longrightarrow\Br(k),\qquad\alpha \longmapsto [c(\sigma,\tau)],
    \]
    which is a group homomorphism by functoriality of the connecting homomorphism in cohomology.  This is the desired map.
\end{proof}

\begin{definition}
We say that $\alpha\in H^1(k,\Pic_{C/k}[n])$ \emph{obstructs the determinantal line bundle} if its image in $\Br(k)$ is nontrivial in the sense of Theorem~\ref{main1}.  Otherwise, we say that $\alpha$ does not obstruct the determinantal line bundle.
\end{definition}

We briefly recall the construction of Kollár~\cite{Kollar_16}, which provides an explicit description of the moduli space of $\alpha$–twisted vector bundles.

\begin{definition}
    Let $L \in \Pic_{X/k}(k)$ be a $k$–point such that $\cL := L \otimes_k K$ is an actual line bundle on $X_K$. The \emph{twisted linear system} associated to $L$, denoted $|L|$, is the irreducible component of $\Hilb_X$ parametrizing subschemes $H \subset X$ for which $H_K$ lies in the linear system $|\cL|$.
\end{definition}

\begin{proposition}\label{brobs}
\textup{\cite[Proposition~69]{Kollar_16}}
    Let $X$ be a projective, geometrically reduced, and geometrically connected $k$–scheme. Then there is a natural exact sequence
    \[
        0 \longrightarrow \Pic(X)\longrightarrow \Pic_{X/k}(k) \xrightarrow{\,b\,} \Br(k)
    \]
    such that $b(L)$ is the Brauer class of the twisted linear system $|L|$.
\end{proposition}

\begin{remark}
    By definition, the twisted linear system $|L|$ is a $k$-form of the linear system $|\cL|$, and therefore is a Brauer--Severi variety. Moreover, if $b(L)=0$, then there is no Brauer obstruction to the class $L \in \Pic_{X/k}(k)$; equivalently, $L$ is represented by an actual line bundle on $X$.  In this case $|L| = |\cL|$, and the corresponding linear system is untwisted.
\end{remark}

\begin{theorem}
Let $X = M^{ss,\alpha}_{C/k}(n,\cL)$.  Over $\bar{k}$ we have $\Pic(X_{\bar{k}})\cong\ZZ$, generated by the determinantal line bundle. Let $L \in \Pic_{X/k}(k)$ be any $k$–point whose base change $L_{\bar{k}}$ maps to the (positive) generator of $\Pic(X_{\bar{k}})$. Denote by $\alpha = b(L)\in\Br(k)$ the Brauer class associated to $L$ by Proposition~\ref{brobs}.  
Then
\[
    \Pic(X) \;\cong\; \per(\alpha)\,\ZZ.
\]
\end{theorem}

\begin{proof}
    As in the proof of Theorem~\ref{main1}, the Galois action on $\Pic(X_{\bar{k}})\cong \ZZ$ is trivial.  Hence the natural map $\Pic_{X/k}(k)\to\Pic(X_{\bar{k}})$ is an isomorphism, and we may identify $\Pic_{X/k}(k)\cong \ZZ$, generated by the geometric determinantal line bundle.

    Since $\Pic(X)\hookrightarrow \Pic_{X/k}(k)\cong\ZZ$, the quotient $\Pic_{X/k}(k)/\Pic(X)$ identifies with $\operatorname{im}(b)$ in the exact sequence
    \[
        0 \longrightarrow \Pic(X)\longrightarrow 
        \Pic_{X/k}(k) \xrightarrow{\,b\,} \Br(k).
    \]

    The image of the generator of $\Pic_{X/k}(k)\cong\ZZ$ under $b$ is the Brauer class $\alpha$, whose order is $\per(\alpha)$. Thus 
    \[
        \operatorname{im}(b)\;\cong\; \ZZ/\per(\alpha) \ZZ.
    \]

    It follows that the subgroup $\Pic(X)$ is exactly the subgroup of index $\per(\alpha)$ in $\ZZ$, i.e.
    \[
        \Pic(X)\;\cong\;\per(\alpha)\,\ZZ.
    \]
\end{proof}

\begin{theorem}\label{coarseisobs}
    Let $\alpha\in H^2(C,\mu_2)$ be such that $\alpha_{\bar{k}} = 0$. Assume $\alpha$ obstructs the determinantal line bundle. If $g(C) = 2$, then the moduli space
    \[
        M^{ss,\alpha}_C(2,\cO_C)\;\cong\;|\Theta|
    \]
    is a $3$-dimensional Brauer--Severi variety.  Further, assume $\operatorname{char} k = 0$ and $g(C)=3$.  Then the morphism
    \[
        \varphi\colon M^{ss,\alpha}_{C/k}(2,\cO_C)\longrightarrow |\Theta|
    \]
    induced by the (twisted) theta linear system is an embedding.
\end{theorem}

\begin{proof}
    Work first over $\bar{k}$.  Since $\alpha_{\bar{k}}=0$, the twisted moduli space and theta data become untwisted, and we are in the classical situation.  In particular, Theorem~\ref{bnr} identifies $M^{ss}_{C_{\bar{k}}/\bar{k}}(2,\cO_{C_{\bar{k}}})$ with the (untwisted) linear system $|\Theta_{\bar{k}}|$ when $g=2$, and Theorem~\ref{geometry} describes the corresponding map
    \[
        \varphi_{\bar{k}}\colon M^{ss}_{C_{\bar{k}}/\bar{k}}(2,\cO_{C_{\bar{k}}}) \longrightarrow |\Theta_{\bar{k}}|
    \]
    in higher genus, showing in particular that it is an embedding when $g=3$ and $\operatorname{char}k=0$.

    Now return to the base field $k$.  The class $\alpha \in H^1(k,\Pic_{C/k}[2])$ is represented by a $1$-cocycle $\sigma \mapsto \cL_\sigma$ with values in $\Pic_{C_{\bar{k}}/\bar{k}}[2]$.Tensoring by $\cL_\sigma$ induces:
    \begin{itemize}
        \item an action of $\Gal(\bar{k}/k)$ on $M^{ss}_{C_{\bar{k}}/\bar{k}}(2,\cO_{C_{\bar{k}}})$, giving the twisted form $M^{ss,\alpha}_C(2,\cO_C)$; 
        \item via Lemma~\ref{linear} and Theorem~\ref{main1}, the corresponding twisted linear system $|\Theta|$ as a form of $|\Theta_{\bar{k}}|$ in the sense of Proposition~\ref{brobs}.
    \end{itemize}
    
    By construction, the map $\varphi_{\bar{k}}$ is defined functorially using the determinantal line bundle $\Theta_{\bar{k}}$ and its sections, so it is compatible with tensoring by $\cL_\sigma$.  Equivalently, $\varphi_{\bar{k}}$ is $\Gal(\bar{k}/k)$–equivariant for the descent data defined by the cocycle $\{\cL_\sigma\}$.  Therefore $\varphi_{\bar{k}}$ descends to a morphism
    \[
        \varphi\colon M^{ss,\alpha}_C(2,\cO_C)\longrightarrow |\Theta|
    \]
    over $k$.
    
    In genus $2$, $\varphi_{\bar{k}}$ is an isomorphism, hence so is $\varphi$; thus $M^{ss,\alpha}_C(2,\cO_C)\cong |\Theta|$, which is a Brauer-Severi threefold.  In genus $3$ under $\operatorname{char}k=0$, $\varphi_{\bar{k}}$ is an embedding by Theorem~\ref{geometry}, and thus its descent $\varphi$ is also an embedding.  This proves the statement.
\end{proof}

\begin{remark} \label{twist2O}
Let $C$ be a genus $2$ curve over $k$. Then the moduli space $M_C^{ss,\alpha}(2,\cO_C)$ is isomorphic to exactly one of the following, depending on the components of the class $\alpha$:
\begin{enumerate}
    \item[(i)] the projective space $\PP^3 \cong |\Theta|$;
    \item[(ii)] a $3$-dimensional Brauer--Severi variety given by the 
    twisted linear system $|\Theta|$;
    \item[(iii)] the intersection of two quadrics in $\PP^5$;
    \item[(iv)] a Galois twist of an intersection of two quadrics in $\PP^5$.
\end{enumerate}

Here (i) follows from Proposition~\ref{brobs}, using that the determinantal line bundle descends when $\alpha$ does not obstruct it. Case (ii) follows from Theorem~\ref{coarseisobs}. Cases (iii) and (iv) follow from the fact that the moduli space of rank~$2$ vector bundles with odd degree determinant is geometrically isomorphic to the intersection of two quadrics in $\PP^5$ \cite[Theorem~4]{Narasimhan_Ramanan_69}, together with Theorem~\ref{Gtwist}.

More generally, classical theory describes the moduli space when the image of $\alpha \in H^2(C,\mu_2)$ in $H^0(k,R^2 f_* \mu_2) \simeq \ZZ/2\ZZ$ is~$1$ for curves of higher genus.  If $C$ is a smooth projective curve of genus $g \ge 2$, then $M^{ss,\alpha}_C(2,\cO_C)$ is (up to twisting) a pencil of quadrics in the sense of \cite[Theorem~1]{Desale_Ramanan_76}.

More precisely, over $\bar{k}$ the moduli space is isomorphic to the variety of $(g-2)$-dimensional linear subspaces of $\PP^{2g+1}$ contained in two quadrics
\[
    Q_1 \;=\; \sum_{i=1}^{2g+2} X_i^2, 
    \qquad 
    Q_2 \;=\; \sum_{i=1}^{2g+2} \omega_i X_i^2,
\]
where the scalars $\omega_i$ correspond to the $2g+2$ branch points of the canonical double covering $C \to \PP^1$. In particular, $M_C^{ss,\alpha}(2,\cO_C)$ is a (possibly twisted) form of this classical complete intersection description.
\end{remark}

\begin{corollary}
    Let $\alpha\in H^{2}(C,\mu_{2})$ represent a $\mu_{2}$-gerbe $\cC\to C$.  Then the moduli space of rank--$2$ vector bundles on $\cC$ with trivial determinant has two disconnected components.

    If $\alpha_{\bar{k}}=0$, then
    \[
        M_{\cC}^{ss}(2,\cO_{C})\;\cong\;|\Theta|\ \cup\ |\Theta_{\alpha}|\;\cong\;\PP^{3}\ \cup\ X,
    \]
    where $X$ is either $\PP^{3}$ or a twisted form of $\PP^{3}$.

    If $\alpha_{\bar{k}}=1$, then
    \[
        M_{\cC}^{ss}(2,\cO_{C})\;\cong\;|\Theta|\ \cup\ X,
    \]
    where $X$ is either the intersection of two quadrics in $\PP^{5}$ or its twisted form.
\end{corollary}

\begin{proof}
    This follows from Theorem~\ref{rank2}, Theorem~\ref{coarseisobs}, and Remark ~\ref{twist2O}.
\end{proof}

\begin{corollary}\label{index}
    Assume that $\alpha\in H^1(k,\Pic_{C/k}[n])$ obstructs the determinantal line bundle. Then the associated Brauer class $\beta\in \Br(k)$ satisfies $\ind(\beta)\mid n^g .$
\end{corollary}

\begin{proof}
    By Proposition~\ref{brobs}, the obstruction $\alpha$ gives rise to a Brauer class $\beta\in\Br(k)$ represented by the twisted linear system $|\Theta|$.  By Theorem~\ref{coarseisobs}, this twisted linear system is a Brauer--Severi variety of dimension $n^g-1$.  A Brauer--Severi variety of dimension $d$ has dividing $d+1$, hence $\ind(\beta)\mid n^g$.
\end{proof}

\section{Period--index problem for two-torsion classes}
In this section we compute the index of the Brauer classes lying in the image of the map $H^1\bigl(k,\Pic_{C/k}\bigr)\rightarrow \Br(k)$ and establish our main result on the period--index problem for curves of genus~$2$. Unless otherwise stated, we assume that $k$ is a field of characteristic $\neq 2$ and that $C$ is a smooth projective curve over $k$ of genus $2$.

\subsection{Index of the moduli}
Assume there exists a nontrivial $2$-torsion line bundle $\xi \in \Pic^{0}(C)[2]$. Associated with $\xi$ is an étale double cover $\pi \colon C' \rightarrow C,$ constructed from the cyclic algebra $\cA := \cO_{C} \oplus \xi.$

We consider the fixed locus of $\xi$ on the moduli space.  For a field $k$, define
\[
    S_{k} \;:=\; \left\{E \in M^{ss}_{C/k}(n,\cL) \ \big|\ E \cong E \otimes \xi\right\}.
\]
In what follows, we restrict attention to the case of rank~$2$ vector bundles with trivial determinant. 

Assume $\bar{k}$ is the algebraic closure of $k$. Then by \cite[Proposition~3.6 and Remark~3.8]{Beauville_Narasimhan_Ramanan_89}, for every vector bundle $E \in S_{\bar{k}}$ there exists a line bundle $\cL_E \in \Pic(C'_{\bar{k}})$, unique up to the involution $\iota^{*}$, such that 
\[
    E \;\cong\; \pi_{*}\cL_E .
\]
Moreover, the line bundles $\cL_E$ that arise in this way are exactly those satisfying the norm condition $\Norm_{C'/C}(\cL_E) = \xi$.

\begin{lemma}\label{prym}
There is a natural bijection
\[
    S_{\bar{k}}\;\xrightarrow{\ \sim\ }\;\Norm_{C'/C}^{-1}(\xi)\big/(\iota^{*}),\qquad E\longmapsto \cL_E,
\]
where $\iota$ is the involution of the double cover $C' \to C$. Under this identification, $S_{\bar{k}}$ corresponds to a rational curve inside the moduli space $M^{ss}_{C_{\bar{k}}}(2,\cO_{C_{\bar{k}}})$ of rank~$2$ vector bundles with trivial determinant.
\end{lemma}

\begin{proof}
    We work over the algebraic closure $\bar{k}$ throughout. Consider the line bundle $\cL_{E}$ corresponding to $E$ via \cite[Proposition~3.6 and Remark~3.8]{Beauville_Narasimhan_Ramanan_89}. By the determinant formula
    \[
        \det(\pi_{*}\cL_E)\;\cong\;\Norm_{C'/C}(\cL_E)\otimes \det(\pi_{*}\cO_{C'}),
    \]
    and in our situation one has $\det(\pi_{*}\cO_{C'})\cong \xi$.  Thus
    \[
        \det(\pi_{*}\cL_E)\;\cong\;\Norm_{C'/C}(\cL_E)\otimes \xi.
    \]
    The condition that $\pi_{*}\cL_E$ have trivial determinant is therefore equivalent to
    \[
        \Norm_{C'/C}(\cL_{E}) \;\cong\; \xi^{-1}.
    \]

    Since $\xi$ is $2$-torsion, we have $\xi^{-1}\cong\xi$, so we are led to the norm condition $\Norm_{C'/C}(\cL_E)\cong \xi$.  The correspondence we described before then identifies $S_{\bar{k}}$ with
    \[
        \Norm^{-1}_{C'/C}(\xi)\big/(\iota^{*}),
    \]
    where $\iota$ is the involution of $C'$.

    Next, $\Norm^{-1}_{C'/C}(\xi)$ is a torsor under the Prym variety
    $\Prym(C'/C)$, and
    \[
        \dim \Prym(C'/C)\;=\;g(C') - g(C)\;=\;(2g-1)-g\;=\;g-1\;=\;1,
    \]
    where we use the Hurwitz theorem \cite[Theorem~IV.2.4]{Hartshorne_77} to compute $g(C') = 2g-1$ for the étale double cover $\pi$. Thus $\Norm^{-1}_{C'/C}(\xi)$ is a torsor under an elliptic curve.

    We now show that $S_{\bar{k}}$ is a rational curve.  Still working over $\bar{k}$, choose $\eta\in\Pic(C)$ with $\eta^{\otimes 2}\simeq\xi$.
    Then
    \[
        \Norm_{C'/C}(\pi^{*}\eta)= \eta^{\otimes 2}= \xi,
    \]
    so $\pi^{*}\eta\in\Norm^{-1}_{C'/C}(\xi)$.  For any $\cL'\in\Norm^{-1}_{C'/C}(\xi)$ we have
    \[
        \Norm_{C'/C}(\cL'\otimes\pi^{*}\eta^{-1})\;\cong\;\Norm_{C'/C}(\cL')\otimes\Norm_{C'/C}(\pi^{*}\eta^{-1})\;\cong\;\xi\otimes\xi^{-1}\;\cong\;\cO_{C},
    \]
    so translation by $\pi^{*}\eta^{-1}$ induces an isomorphism
    \[
        t_{\pi^{*}\eta^{-1}}\colon\Norm^{-1}_{C'/C}(\xi)\xrightarrow{\;\sim\;}\Prym(C'/C),
    \]
    with inverse given by $\cM\mapsto\cM\otimes\pi^{*}\eta$ for $\cM\in\Prym(C'/C)$.

    Furthermore,
    \[
        t_{\pi^{*}\eta^{-1}}(\iota^{*}\cL)\;=\;\iota^{*}\cL\otimes\pi^{*}\eta^{-1}\;\cong\;\iota^{*}\cL\otimes\iota^{*}\pi^{*}\eta^{-1}\;\cong\;\iota^{*}(\cL\otimes\pi^{*}\eta^{-1}),
    \]
    so under $t_{\pi^{*}\eta^{-1}}$ the involution $\iota^{*}$ on $\Norm^{-1}_{C'/C}(\xi)$ corresponds to the involution $\iota^{*}$ on $\Prym(C'/C)$.  Since $\Prym(C'/C)$ is the anti-invariant part of $\Pic^{0}(C')$ under $\iota^{*}$, the restriction of $\iota^{*}$ to $\Prym(C'/C)$ is the $(-1)$-map.  Thus
    \[
        \Norm^{-1}_{C'/C}(\xi)\big/(\iota^{*})\;\cong\;\Prym(C'/C)\big/\{\pm 1\}.
    \]

    The fixed points of the involution $\{\pm 1\}$ are exactly the $2$-torsion points $\Prym(C'/C)[2]$, of which there are four.  Hence the quotient map
    \[
        \Prym(C'/C) \longrightarrow \Prym(C'/C)\big/\{\pm 1\}
    \]
    is a double cover branched at four points.  Applying the Riemann--Hurwitz formula to this map, we obtain
    \[
        2g\bigl(\Prym(C'/C)\bigr)-2 \;=\; 2\bigl(2g(S_{\bar{k}})-2\bigr) + 4.
    \]
    Since $\Prym(C'/C)$ is an elliptic curve, $g(\Prym(C'/C))=1$, so the left-hand side is $0$ and we get $0 \;=\; 4g(S_{\bar{k}}),$ which forces $g(S_{\bar{k}})=0$.  Thus $S_{\bar{k}}$ is a rational curve.
\end{proof}

\begin{lemma}\label{coprime}
    \textup{\cite[Proposition~4.1.1]{Lieblich_08}}
    Let $K$ be a field and let $\alpha \in \Br(K)$ be a Brauer class of period $\per(\alpha)=n$. If $L/K$ is a finite field extension of degree $d$ with $\gcd(d,n)=1$, then
    \[
        \per(\alpha) \;=\; \per(\alpha|_{L})\qquad\text{and}\qquad\ind(\alpha) \;=\; \ind(\alpha|_{L}).
    \]
\end{lemma}

\begin{theorem}
    Let $k$ be a field such that $H^{1}(k,\Pic_{C/k}[2])$ contains a class obstructing the determinantal line bundle on the moduli space $M^{ss,\alpha}_{C/k}(2,\cO_C)$ of $\alpha$–twisted rank-$2$ vector bundles with trivial determinant. Let $\beta$ be the image of this class under the map $H^{1}(k,\Pic_{C/k}[2]) \rightarrow \Br(k)$ constructed in Theorem~\ref{main1}. Then $\per(\beta)=\ind(\beta)=2.$
\end{theorem}

\begin{proof}
    Since $\Pic^{0}(C)[2]$ is finite étale of rank $16$ (as $\operatorname{char}(k)\neq 2$), every nontrivial $2$-torsion line bundle $\xi$ becomes defined over a finite separable extension whose degree equals the size of its Galois orbit.  Among the possible orbit sizes, at least one nontrivial $2$-torsion point has odd orbit size, since the $15$ nonzero points cannot all lie in even-sized orbits.  Since  $\Pic^{0}(C)[2]$ is finite étale in all characteristics $\neq 2$, such an odd-degree field of definition may always be obtained by a finite separable extension.  Thus, after a finite separable extension $L/k$ of odd degree, we may assume that a chosen nontrivial $\xi$ is defined over $L$.

    Because the period of the Brauer class $\beta$ is $2$, Lemma~\ref{coprime} implies that period and index are unchanged under odd-degree separable extensions.  Hence we may assume from the outset that $\xi$ is defined over $k$.

    By Proposition~\ref{brobs}, the obstruction class $\beta\in\Br(k)$ is represented by the Brauer--Severi threefold $|\Theta|\cong M^{ss,\alpha}_{C/k}(2,\cO_C)$.  Since we have exhibited a Brauer--Severi curve $S_k \subset M^{ss,\alpha}_{C/k}(2,\cO_C)$, the index of $\beta$ equals the dimension of this minimal twisted linear subvariety plus one.  Thus $\ind(\beta)=2$, and therefore $\per(\beta)=2$ as well.
\end{proof}

\begin{theorem}\label{main5}
    Assume $\alpha\in \Br(C)[2]$ is the Brauer class whose image in $H^{1}(k,\Pic_{C/k})$ obstructs the determinantal line bundle on $M^{ss,\alpha}_{C/k}(2,\cO_C)$.  Assume further that for every finite extension $K/k$ and every $\beta\in\Br(K)$ one has $\ind(\beta)\mid 2^{\,i}$. Then $\ind(\alpha)\mid 2^{\,i+2}$.
\end{theorem}

\begin{proof}
    Let $K/k$ be a quadratic extension that splits the Brauer--Severi curve $S_k \subset M^{ss,\alpha}_{C/k}(2,\cO_C)$ constructed in Lemma~\ref{prym}.  Thus $S_k(K)\neq\varnothing$, and hence $M^{ss,\alpha}_{C/k}(2,\cO_C)(K)\neq\varnothing$.
    
    A $K$--rational point of the coarse moduli space lifts to a $K$--point of the moduli stack of $\alpha$-twisted sheaves if and only if the associated $\mu_{2}$-gerbe is trivial over $K$.  The obstruction is a $2$-torsion Brauer class $\beta\in\Br(K)[2]$, and by hypothesis $\ind(\beta)\mid 2^{\,i}$. Thus, after a further extension $L/K$ of degree dividing $2^{\,i}$, the class $\beta$ becomes trivial and the gerbe is trivial over $L$.
    
    Over $L$, the point of the coarse moduli space lifts to an $\alpha$--twisted vector bundle of rank~$2$ with trivial determinant. To trivialize $\alpha_L$, we need an additional degree-$2$ extension, and this contributes one more factor of $2$.
    
    Combining these extensions, we obtain an extension of $k$ of degree dividing $2^{\,i+2}$ over which $\alpha$ becomes split.  Hence $\ind(\alpha)\mid 2^{\,i+2}$, as claimed.
\end{proof}

\begin{remark}
    A natural idea for reducing the index is to work inside the strictly semistable locus.  Suppose $K/k$ is a quadratic extension such that $S_k\otimes K$ splits.  Let $E \in M^{ss,\alpha}_C(2,\cO_C)(K)$ be a strictly semistable bundle. Then $E$ is $S$--equivalent to a direct sum $E \cong \cL \oplus \cL^{\vee}$ for some line bundle $\cL$ on $C_K$.
    
    We ask whether such points can lie on $S_K$, i.e.\ whether $E \;\cong\; E\otimes \xi$. Since
    \[
    E\otimes\xi \;\cong\;(\cL\otimes\xi)\,\oplus\,(\cL^{\vee}\otimes\xi)\quad\text{and}\quad E\cong \cL\oplus\cL^{\vee},
    \]
    this requires $\cL \;\cong\; \cL^{\vee}\otimes \xi$, hence $\cL^{\,2}\;\cong\; \xi$. Thus the intersection $S_K \cap M^{ss,\alpha}_C(2,\cO_C)^{\mathrm{strict}}$ consists exactly of bundles of the form $\cL\oplus\cL^{\vee}$ with $\cL^{2}\cong\xi$.
    
    Geometrically, the intersection of $S_{K}$ with the strictly semistable locus is described by the line bundles $\cL$ and $\cL^{\vee}$ satisfying $\cL^{\otimes 2}\cong \xi$.  Hence this intersection maps to a $K$--rational point on the twisted Kummer surface in $M^{ss,\alpha}_C(2,\cO_C)$. 
    
    However, this does not reduce the index. To reduce the index one would need a $K$--rational point of the twisted Picard variety $\Pic^0_{C,\alpha}$, which is a double cover of its twisted Kummer surface away from the $2$--torsion locus. A $K$--rational point on the Kummer surface does \emph{not} necessarily lift to a $K$--rational point of $\Pic^0_{C,\alpha}$; in general its two preimages are defined only over a quadratic extension of~$K$. Thus this approach cannot produce a further drop in the index.
    
    In fact, the following example shows that such a lift need not exist.
\end{remark}

\begin{theorem}\label{C1field}
    Let $k$ be a $C_{1}$ field and $C/k$ a genus~$2$ curve.  Then for every nontrivial $\alpha\in\Br(C)[2]$ one has
    \[
        \per(\alpha)=\ind(\alpha)=2.
    \]
\end{theorem}
\begin{proof}
    Let $\alpha \in \Br(C)[2]$ be nontrivial. Since $k$ is $C_{1}$, we have $\Br(k)=0$, and the Leray spectral sequence gives an isomorphism
    \[
        \Br(C)\;\cong\;H^{1}\!\left(k,\Pic_{C/k}\right).
    \]
    Thus $\alpha$ corresponds to a nontrivial element of $H^{1}(k,\Pic_{C/k})$.

    By Proposition~\ref{brobs}, the vanishing of $\Br(k)$ implies that $\alpha$ does not obstruct the determinantal line bundle of $M^{ss,\alpha}_{C/k}(2,\cO_{C})$. Hence this moduli space is isomorphic to $\PP^{3}$, the untwisted linear system $|\Theta|$, and in particular it contains a $k$--rational point. Since $\Br(k)=0$ by Tsen's theorem and $\alpha$ is nontrivial, the proof of Theorem~\ref{main5} shows that the index is only contributed by the existence of a rank~$2$ twisted vector bundle with trivial determinant, yielding $\ind(\alpha) = 2$. 
\end{proof}

\begin{remark}
    In the proof of Theorem~\ref{C1field}, the fact that $\ind(\alpha)=2$ also shows that the lift of the $k$--rational point from the twisted Kummer surface to the twisted Picard variety is impossible.
\end{remark}

\begin{remark}
    Throughout this section, the term index always refers to the index of a Brauer class.  
    
    We emphasize that the index of a $\Pic_{C/k}$–torsor may be considerably larger; see, for example, \cite[Lemma~7.6]{Creutz_Viray_23}.
\end{remark}

\subsection{The period–index problem for certain $n$–folds}
In this subsection we construct a class of $n$–dimensional varieties obtained as iterated fibre products of genus~$2$ curves.  These examples carry natural $2$–torsion Brauer classes whose period–index behaviour can be analysed using the results of the previous sections.

\begin{lemma}\label{prodjacobian}
    Let $X$ and $Y$ be smooth projective varieties over $k$, where $\cha k = 0$. Then there is an isomorphism
    \[
        \Pic^{0}_{X\times Y/k}\;\cong\;\Pic^{0}_{X/k}\times \Pic^{0}_{Y/k}.
    \]
\end{lemma}

\begin{proof}
    By \cite[Theorem~3.3.12]{Smith_06}, there is an isomorphism
    \[
        \Pic_{X\times Y/k}\;\cong\;\Pic_{X/k}\times \Pic_{Y/k}\times \Hom\!\bigl(\Pic^{0}_{X/k},\Pic^{0}_{Y/k}\bigr).
    \]
    By \cite[Cor.~1, p.~178]{Mumford_74}, for any abelian varieties $A$ and $B$ one has
    \[
        \Hom(A,B)\cong \ZZ^{\rho},\textrm{ where }\rho\le 4\dim A\dim B,
    \]
    so the group $\Hom(\Pic^{0}_{X/k},\Pic^{0}_{Y/k})$ is discrete.

    Passing to connected components of the identity, the discrete factor contributes nothing. Thus we have
    \[
        \Pic^{0}_{X\times Y/k}
        \;\cong\;
        \Pic^{0}_{X/k}\times \Pic^{0}_{Y/k}.
    \]
\end{proof}

\begin{remark}
    An alternative proof of the same result appears in \cite[Proposition~5.7.1]{Colliot-Thelene_Skorobagatov_21}. Their argument does not require the assumption $\operatorname{char} k = 0$, even though $\Pic^0_{X/k}$ may fail to be an abelian variety in positive characteristic due to nonreducedness. Consequently, we do not impose this assumption in the following theorem.
\end{remark}

\begin{definition}
    The \emph{algebraic Brauer group} of $X$, denoted $\Br_1(X)$, is defined as
    \[
        \Br_1(X) := \ker\!\bigl(\Br(X) \longrightarrow H^0(k, R^2 f_* \GG_m)\bigr)\subset \Br(X).
    \]
\end{definition}

\begin{theorem}
    Let $k$ be a $C_{1}$ field, and let 
    \[
        X = C_{1}\times \cdots \times C_{n}
    \]
    be the product of $n$ genus $2$ curves over $k$. Let $\alpha \in \Br_1(X)$ be an algebraic Brauer class. Then the period-index bound holds for $\alpha$, namely
    \[
        \ind(\alpha)\mid \per(\alpha)^{\,n}=2^{n}.
    \]
\end{theorem}

\begin{proof}
    Since $k$ is $C_{1}$, we have $\Br(k)=0$. The Leray spectral sequence for $X\to \Spec k$ therefore gives
    \[
        \Br_1(X)[2]\;\cong\;H^{1}\!\left(k,\Pic_{X/k}\right)[2].
    \]
    In particular, every $2$-torsion class on $X$ arises from $H^{1}(k,\Pic_{X/k}[2])$.

    By Lemma~\ref{prodjacobian}, the identity component of the Picard scheme of the product splits:
    \[
        \Pic^{0}_{X/k}\;\cong\;\Pic^{0}_{C_{1}/k}\times \cdots \times \Pic^{0}_{C_{n}/k}.
    \]
    Taking $2$-torsion and using that $[2]$ is a finite morphism of group schemes, we obtain
    \[
        \Pic^{0}_{X/k}[2]\;\cong\;\Pic^{0}_{C_{1}/k}[2]\times \cdots \times \Pic^{0}_{C_{n}/k}[2].
    \]

    Since Galois cohomology commutes with finite direct products, we get
    \[
        H^{1}\!\left(k,\Pic^{0}_{X/k}[2]\right)\;\cong\;\bigoplus_{i=1}^{n}H^{1}\!\left(k,\Pic^{0}_{C_{i}/k}[2]\right).
    \]
    Thus any class $\alpha\in H^{1}\left(k,\Pic^{0}_{X/k}[2]\right)$ decomposes uniquely as
    \[
        \alpha = (\alpha_{1},\dots,\alpha_{n}),\textrm{ for }\alpha_{i}\in H^{1}\!\left(k,\Pic^{0}_{C_{i}/k}[2]\right).
    \]

    By Theorem~\ref{C1field}, each nontrivial $\alpha_{i}$ satisfies
    \[
        \ind(\alpha_{i})=\per(\alpha_{i})=2.
    \]
    The index of $\alpha$ is at most the product of the indices of the nontrivial components~$\alpha_{i}$. Therefore
    \[
        \ind(\alpha)\;\mid\;2^{\,n}.
    \]
    This proves the desired period-index bound.
\end{proof}

\section*{Conflict of interest and ethical statement}
On behalf of all authors, the corresponding author states that there is no ethical issues or conflict of interest associated with this manuscript. No human participants or animals were involved in this research.

\printbibliography
\end{document}